\documentclass[final,12pt,a4paper]{amsart}
    \title{Functorially finite hearts, simple-minded systems in negative cluster categories, and noncrossing partitions}

\author{Raquel Coelho Sim\~oes, David Pauksztello and David Ploog \\ with an appendix by Raquel Coelho Sim\~oes, David Pauksztello and Alexandra Zvonareva}

\usepackage{amssymb}
\usepackage{latexsym}
\usepackage{amsmath}
\usepackage{mathrsfs}
\usepackage{mathtools}
\usepackage{dsfont}
\usepackage[all]{xy}
\usepackage{enumitem}
\usepackage{graphicx}
\usepackage{fullpage}

\usepackage[usenames,dvipsnames]{xcolor}
\usepackage{hyperref}
\hypersetup{
    colorlinks,
    linkcolor={red!50!black},
    citecolor={blue!50!black},
    urlcolor={blue!80!black}
}

\newcommand{\harxiv}[1]{\href{http://arxiv.org/abs/#1}{\texttt{arXiv:#1}}}
\newcommand{\hyref}[2]{ \hyperref[#2]{#1~\ref*{#2}} }

\usepackage[notref, notcite]{showkeys} 

\usepackage{color}

\theoremstyle{plain}
\newtheorem{theorem}{Theorem}[section]
\newtheorem{lemma}[theorem]{Lemma}
\newtheorem{corollary}[theorem]{Corollary}
\newtheorem{proposition}[theorem]{Proposition}
\newtheorem{introtheorem}{Theorem}

\theoremstyle{definition}
\newtheorem{remark}[theorem]{Remark}
\newtheorem{example}[theorem]{Example}
\newtheorem{definition}[theorem]{Definition} 
\newtheorem{definitions}[theorem]{Definitions} 
\newtheorem{setup}[theorem]{Setup}

\newcommand{\sC}{\mathsf{C}}
\newcommand{\sD}{\mathsf{D}}
\newcommand{\sE}{\mathsf{E}}
\newcommand{\sH}{\mathsf{H}}
\newcommand{\sP}{\mathsf{P}}
\newcommand{\sR}{\mathsf{R}}
\newcommand{\sS}{\mathsf{S}}
\newcommand{\sT}{\mathsf{T}}
\newcommand{\sU}{\mathsf{U}}
\newcommand{\sV}{\mathsf{V}}
\newcommand{\sX}{\mathsf{X}}
\newcommand{\sY}{\mathsf{Y}}
\newcommand{\sZ}{\mathsf{Z}}

\newcommand{\su}{\mathbf{u}}

\newcommand{\Db}{\mathsf{D}^b}
\newcommand{\Kb}{\mathsf{K}^b}
\newcommand{\Cw}{\sC_{-w}}
\newcommand{\Cm}{\sC_m}

\newcommand{\cC}{\mathcal{C}}
\newcommand{\cE}{\mathcal{E}}
\newcommand{\cF}{\mathcal{F}}
\newcommand{\Fw}{\cF_{-w}}
\newcommand{\ES}{\cE_\sS}

\newcommand{\bP}{\mathbb{P}}
\newcommand{\bZ}{\mathbb{Z}}

\renewcommand{\geq}{\geqslant}
\renewcommand{\leq}{\leqslant}
\renewcommand{\phi}{\varphi}
\renewcommand{\epsilon}{\varepsilon}

\DeclareMathOperator{\NC}{\mathsf{NC}}
\DeclareMathOperator{\Red}{\mathsf{red}}
\DeclareMathOperator{\Hom}{\mathsf{Hom}}
\DeclareMathOperator{\Ext}{\mathsf{Ext}}

\newcommand{\add}[1]{\mathsf{add}(#1)}
\newcommand{\ind}[1]{\mathsf{ind}(#1)}
\renewcommand{\mod}[1]{\mathsf{mod}(#1)}
\newcommand{\proj}[1]{\mathsf{proj}(#1)}
\newcommand{\coh}[1]{\mathsf{coh}(#1)}
\newcommand{\Qcoh}[1]{\mathsf{Qcoh}(#1)}
\newcommand{\inj}[1]{\mathsf{inj}(#1)}
\newcommand{\Mod}[1]{\mathsf{Mod}(#1)}

\newcommand{\supp}[1]{\mathsf{supp}(#1)}

\newcommand{\kk}{{\mathbf{k}}}

\newcommand{\SSS}{\mathbb{S}}

\newcommand{\extn}[1]{\langle #1 \rangle}
\newcommand{\extnZ}[1]{\{ #1 \}}

\newcommand{\orth}{^\perp}

\newcommand{\orthH}{^{\perp_{\sH}}}

\DeclareMathOperator{\susp}{\mathsf{susp}}
\DeclareMathOperator{\cosusp}{\mathsf{cosusp}}
\newcommand{\thick}[2]{\mathsf{thick}_{#1}(#2)}

\newcommand{\tA}[1]{\tilde{A}_{#1}}

\newcommand{\too}{\longrightarrow}
\newcommand{\rightlabel}[1]{\stackrel{#1}{\longrightarrow}}
\newcommand{\bij}{\stackrel{1-1}{\longleftrightarrow}}
\newcommand{\into}{\hookrightarrow}
\newcommand{\onto}{\twoheadrightarrow}

\newcommand{\tri}[3]{#1\rightarrow #2\rightarrow #3\rightarrow \Sigma #1}
\newcommand{\trilabel}[4]{#1\stackrel{#4}{\longrightarrow} #2\longrightarrow #3\longrightarrow \Sigma #1}
\newcommand{\trilabels}[6]{#1\stackrel{#4}{\longrightarrow} #2\stackrel{#5}{\longrightarrow} #3\stackrel{#6}{\longrightarrow} \Sigma #1}
   
\newenvironment{pmat}{\left[ \begin{smallmatrix}}{\end{smallmatrix} \right]}

\entrymodifiers={+!!<0pt,\fontdimen22\textfont2>}

\newcommand{\coloneqq}{\mathrel{\mathop:}=}

\usepackage{tikz}
\usetikzlibrary{matrix,arrows,backgrounds,shapes.misc,shapes.geometric,patterns,calc,positioning}
\usetikzlibrary{calc,shapes,fadings}
\usetikzlibrary{decorations.pathmorphing,automata}
\usetikzlibrary{decorations.pathreplacing}
\tikzstyle arrowstyle=[scale=1]
\tikzstyle directed=[postaction={decorate,decoration={markings,
    mark=at position .65 with {\arrow[arrowstyle]{stealth}}}}]
\tikzstyle reverse directed=[postaction={decorate,decoration={markings,
    mark=at position .65 with {\arrowreversed[arrowstyle]{stealth};}}}]
\tikzstyle{box} = [rectangle, draw]
\tikzstyle{line} = [draw, -latex']

\makeatletter
\@namedef{subjclassname@2020}{%
  \textup{2020} Mathematics Subject Classification}
\makeatother

\begin{document}

\begin{abstract}
Let $Q$ be an acyclic quiver and $w \geq 1$ be an integer. Let $\Cw(\kk Q)$ be the $(-w)$-cluster category of $\kk Q$. We show that there is a bijection between simple-minded collections in $\Db(\kk Q)$ lying in a fundamental domain of $\Cw(\kk Q)$ and $w$-simple-minded systems in $\Cw(\kk Q)$. This generalises the same result of Iyama-Jin in the case that $Q$ is Dynkin.
 A key step in our proof is the observation that the heart $\sH$ of a bounded t-structure in a Hom-finite, Krull-Schmidt, $\kk$-linear saturated triangulated category $\sD$ is functorially finite in $\sD$ if and only if $\sH$ has enough injectives and enough projectives.
We then establish a bijection between $w$-simple-minded systems in $\Cw(\kk Q)$ and positive $w$-noncrossing partitions of the corresponding Weyl group $W_Q$.
\end{abstract}



\maketitle

{\small
\setcounter{tocdepth}{1}
\tableofcontents
}

\section*{Introduction} 

Cluster-tilting theory was introduced in \cite{BMRRT} as an approach to categorifying the cluster algebras of Fomin and Zelevinsky in \cite{FZ}. Since its inception, cluster-tilting theory has gone on to have widespread connections with many areas of mathematics. 

Classical cluster-tilting theory takes place in an $m$-cluster category $\Cm(\kk Q)$, for $m \geq 2$, which is an $m$-Calabi-Yau orbit category of the bounded derived category, $\Db(\kk Q)$, of the path algebra of a finite acyclic quiver $Q$, where $\kk$ is an algebraically closed field.
One of the most fruitful connections has been between the representation theory of finite-dimensional algebras and Coxeter combinatorics. 
For example, let $W$ be the Weyl group of type $Q$. There are natural bijections between the sets of clusters and noncrossing partitions associated to $W$; see \cite{ABMW,Reading,STW}. 
There is in turn a bijection between the sets of clusters and cluster-tilting objects in the corresponding cluster category \cite{BRT}.
Beyond cluster-tilting theory, there are many further connections between the combinatorics of noncrossing partitions and representation theory, e.g.\ in the classification of wide subcategories and torsion theories \cite{Ingalls-Thomas}, thick subcategories \cite{Kohler} and Cartan lattices \cite{HK}, to name a few. For a broad treatment of the combinatorics of noncrossing partitions in finite type we refer the reader to \cite{STW}.

Recently, there has been increasing interest in negative Calabi-Yau triangulated categories; see, for example, \cite{Brightbill1,Brightbill2,CS12, CS15, CS17, CSP16, CSP20, Dugas, HJY, IJ, Jin1, Jin2, Koenig-Liu}. 
For an integer $w \geq 1$ and an acyclic quiver $Q$ there is an orbit category $\Cw(\kk Q)$ of the bounded derived category $\Db(\kk Q)$ which is $(-w)$-Calabi-Yau. This orbit category can be thought of as a `negative (Calabi-Yau) cluster category'. 
In this setting the analogue of $m$-cluster-tilting objects are so-called $w$-simple-minded systems. 
Evidence supporting the viewpoint that $w$-simple-minded systems are a negative Calabi-Yau analogue of cluster-tilting objects is advanced by a growing body of work: \cite{BRT, CS12, CS15, CS17, CSP20, Dugas, IJ, Jin1,Jin2}.

It is therefore natural to ask about connections between negative cluster categories and Coxeter combinatorics. Previous work in this direction includes \cite{BRT, CS12, IJ,STW}. Again let $Q$ be an acyclic quiver with corresponding Weyl group $W$. In \cite{CS12}, the first author established a bijection between so-called positive noncrossing partitions of $W$ and \mbox{($1$-)}simple-minded systems in $\sC_{-1}(\kk Q)$ when $Q$ is Dynkin. In \cite{BRT}, Buan, Reiten and Thomas obtained a bijection between simple-minded collections lying inside some `fundamental domain' and $w$-noncrossing partitions, which was a forefather of the K\"onig-Yang correspondences \cite{KY}. Most recently, in \cite{IJ}, Iyama and Jin generalised \cite{BRT} and \cite{CS12} to obtain a bijection between $w$-simple-minded systems in $\Cw(\kk Q)$ and positive $w$-noncrossing partitions, again for $Q$ Dynkin. This bijection proceeds via a bijection between $w$-simple-minded systems in $\Cw(\kk Q)$ and simple-minded collections lying in the fundamental domain $\Fw$ (see Section~\ref{sec:hereditary} for the precise definition) of $\Db(\kk Q)$; see \cite[Theorem 1.2]{IJ}.

In this article we extend this bijection to the case that $Q$ is an arbitrary quiver in the following main theorem.

\begin{introtheorem}[Theorem~\ref{thm:SMC-SMS-bijection}] \label{thm:bijection}
Let $Q$ be an acyclic quiver. The natural projection functor $\pi \colon \Db(\kk Q) \to \Cw(\kk Q)$ induces a bijection
\[ 
\left\{ 
\parbox{6.75cm}{\centering
simple-minded collections of $\Db(\kk Q)$ contained in $\Fw$}
\right\} 
\bij 
\left\{ 
\parbox{4.75cm}{\centering
$w$-simple-minded systems in $\Cw$}
\right\}. 
\]
\end{introtheorem}

Using Theorem~\ref{thm:bijection}, we are then able to obtain the following bijection involving noncrossing partitions for any acyclic quiver $Q$.

\begin{introtheorem}[Theorem~\ref{thm:noncrossing}] \label{thm:nc}
Let $Q$ be an acyclic quiver. There is a bijection
\[
\left\{
\parbox{4.25cm}{\centering
positive $w$-noncrossing partitions of $W_Q$}
\right\}
\bij
\left\{ 
\parbox{4.75cm}{\centering
$w$-simple-minded systems in $\Cw(\kk Q)$}
\right\}. 
\]
\end{introtheorem}

In order to obtain Theorem~\ref{thm:bijection}, we require another observation that we believe holds independent interest and will be widely applicable. Let $\sD$ be a Hom-finite, Krull-Schmidt, $\kk$-linear triangulated category and suppose that $\sH \subseteq \sD$ is the heart of a bounded t-structure. One can ask what is the relationship between the following properties of $\sH$:
\begin{itemize}
\item $\sH$ is a length category with finitely many simple objects;
\item $\sH$ has enough injective objects and enough projective objects; and,
\item $\sH$ is a functorially finite subcategory of $\sD$.
\end{itemize}
Of these three properties, the final property is the odd one out: it is the only property which takes the ambient triangulated category in which $\sH$ sits into account. As such a relationship between these properties is potentially very powerful. Unfortunately, Example~\ref{ex:not-funct-finite} shows that there is no relationship, in general, between the first and the third properties. However, our third main theorem provides a relationship between the second and third properties of $\sH$. It is this relationship which is a crucial ingredient in our proof of Theorem~\ref{thm:bijection}.

\begin{introtheorem}[Corollary~\ref{cor:saturated}] \label{thm:functorially-finite}
Let $\sD$ be a Hom-finite, Krull-Schmidt, saturated triangulated category. Suppose $\sH \subseteq \sD$ is the heart of a bounded t-structure. Then $\sH$ is functorially finite in $\sD$ if and only if $\sH$ has enough injectives and enough projectives.
\end{introtheorem}

In this statement we have specialised to the case that $\sD$ is a saturated triangulated category, for example, $\sD = \Db(A)$ for a finite-dimensional $\kk$-algebra $A$ of finite global dimension or $\sD = \Db(\coh{X})$ for a smooth projective variety $X$. A more general, technical statement without the saturated hypothesis is proved in Theorem~\ref{thm:heart-finite}.

We briefly sketch the structure of the paper. In Section~\ref{sec:prelim}, we recall the basic concepts we use in the paper. In Section~\ref{sec:funct-finite} we prove Theorem~\ref{thm:functorially-finite}. Section~\ref{sec:simple-minded} recalls the basic properties of orthogonal collections, simple-minded systems and simple-minded collections, and establishes a characterisation of simple-minded collections in terms of Riedtmann configurations that may be more widely applicable and illustrates the parallel between simple-minded collections and simple-minded systems. In Section~\ref{sec:bijection}, we prove our main theorem, Theorem~\ref{thm:bijection}.
Section~\ref{sec:sincere} establishes a bijection between $w$-simple-minded systems and certain sincere orthogonal collections which is the crucial tool to pass from Theorem~\ref{thm:bijection} to Theorem~\ref{thm:nc}, which is done in Section~\ref{sec:noncrossing}.
Finally, the paper includes an appendix by the first two authors and Alexandra Zvonareva in which we give an alternative proof of a theorem by Jin \cite{Jin2} on the reduction of simple-minded collections using the characterisation in terms of Riedtmann configurations, which avoids the passage to a Verdier localisation, and is more in the spirit of \cite{CSP20}.

\subsection*{Notation convention}
In abstract abelian and triangulated categories we will use lower case Roman letters to denote objects. When we specialise to module categories or derived categories of an algebra, we will use upper case to denote objects. The philosophy behind this is that in the latter case, these objects have elements.

\subsection*{Acknowledgments}
We would like to thank Lidia Angeleri H\"ugel, Jorge Vit\'oria and Alexandra Zvonareva for useful discussions. 
We particularly thank Alexandra Zvonareva for kindly allowing us to add her joint work with the first two authors as an appendix.
The authors are grateful to Peter J\o rgensen and Haruhisa Enomoto for pointing out a gap in our original proof of Theorem~\ref{thm:functorially-finite}, and an anonymous referee for useful comments and suggestions.
The first author is grateful to the European Union's Horizon 2020 research and innovation programme for financial support through the Marie Sk\l odowska-Curie Individual Fellowship grant agreement number 838706.

\section{Preliminaries} \label{sec:prelim}

To begin with $\sD$ will be a Hom-finite, Krull-Schmidt, $\kk$-linear triangulated category over a field $\kk$, where \emph{Hom-finite} means that $\Hom_\sD(x,y)$ is finite-dimensional as a $\kk$-vector space for any objects $x$ and $y$ in $\sD$. Starting from Section~\ref{sec:simple-minded}, $\kk$ will be assumed to be algebraically closed. The shift or suspension functor will be denoted by $\Sigma \colon \sD \to \sD$. Later we will specialise to the case that $\sD = \Db(\kk Q)$ for a finite acyclic quiver $Q$. Abusing notation, if $\sX$ is a collection of objects of $\sD$ and $d$ is an object of $\sD$ we will write $\Hom_\sD(\sX,d)$ to mean $\Hom_\sD(x,d)$ where we take each $x \in \sX$ in turn; likewise for $\Hom_\sD(d,\sX)$. We shall assume all subcategories are full and strict.

For subcategories $\sX$ and $\sY$ of $\sD$, we write
\[
\sX * \sY = \{ d \in \sD \mid \text{there exists a triangle } \tri{x}{d}{y} \text{ with } x \in \sX \text{ and } y \in \sY\}.
\]
We note that by the octahedral axiom, the $*$ product of subcategories is associative. A subcategory $\sX$ is \emph{extension-closed} if $\sX * \sX = \sX$. We denote by $\extn{\sX}$, or sometimes by $\extn{\sX}_\sD$ when we need to emphasise the triangulated category in which we are working, the \emph{extension closure} of $\sX$, that is the smallest extension-closed subcategory of $\sD$ containing $\sX$.
The \emph{right and left perpendicular categories} of $\sX$ are defined as follows:
\[
 \sX\orth = \{ d \in \sD \mid \Hom(\sX,d) = 0 \}
 \text{ and }
 {}\orth \sX = \{ d \in \sD \mid \Hom(d,\sX) = 0 \}.
 \]
In Section~\ref{sec:sincere} we will also require a notion of perpendicular category that is more suited to abelian categories; see Definition~\ref{def:perp}.
 
An autoequivalence $\SSS \colon \sD \to \sD$ is called a Serre functor if for each $x,y \in \sD$ there is an isomorphism,
 \[
 \Hom(x,y) \simeq D\Hom(y, \SSS x),
 \]
which is natural in $x$ and $y$, where $D = \Hom_\kk(-,\kk)$. If $\sD$ has a Serre functor, it is unique up to isomorphism and we say $\sD$ satisfies \emph{Serre duality}. For details we refer to \cite{RvdB}.

Let $w\in \bZ$. A triangulated category $\sD$ satisfying Serre duality is \emph{$w$-Calabi-Yau} (or \emph{$w$-CY}) if there is a natural isomorphism $\SSS \simeq \Sigma^w$, where $\SSS$ is the Serre functor on $\sD$.

\subsection{Functorially finite subcategories and (co-)t-structures}
Let $\sX$ be a subcategory of $\sD$, and $d$ an object in $\sD$. A morphism $f \colon x \to d$, with $x \in \sX$, is a \emph{right $\sX$-approximation} (or an \emph{$\sX$-precover}) of $d$ if $\Hom(\sX, f) \colon \Hom(\sX,x) \to \Hom(\sX,d)$ is surjective.
The morphism $f$ is called \emph{right minimal} if any $g \colon x \to x$ satisfying $fg = f$ is an automorphism.
The morphism $f$ is called a \emph{minimal right $\sX$-approximation} (or an \emph{$\sX$-cover}) of $d$ if it is a right $\sX$-approximation of $d$ and is right minimal.
If $f \colon x \to d$ is a minimal right $\sX$-approximation and $h \colon x' \to d$ is a right $\sX$-approximation then $x$ is isomorphic to a direct summand of $x'$; e.g.\ \cite{AS}.

If every object in $\sD$ admits a right $\sX$-approximation, then $\sX$ is said to be \emph{contravariantly finite} or \emph{precovering}. If every object in $\sD$ admits a minimal right $\sX$-approximation, then $\sX$ is said to be \emph{covering}. 
There are dual notions of \emph{(minimal) left $\sX$-approximations} (or \emph{$\sX$-pre-envelopes} and \emph{$\sX$-envelopes}) and \emph{covariantly finite} (or \emph{(pre-)enveloping}) subcategories. The subcategory $\sX$ of $\sD$ is called \emph{functorially finite} if it is both contravariantly finite and covariantly finite.
We note that if $\sD$ is Hom-finite and Krull-Schmidt then any precovering (resp.\ pre-enveloping) subcategory is automatically covering (resp.\ enveloping); e.g.\ \cite{AS}.

\begin{definition}
A pair of full subcategories $(\sX,\sY)$ of $\sD$, each closed under summands, such that $\Hom_\sD(\sX,\sY) = 0$ and $\sD = \sX * \sY$ is called a
\begin{itemize}
\item \emph{t-structure} if in addition $\Sigma \sX \subseteq \sX$ (equivalently, $\Sigma^{-1} \sY \subseteq \sY$). 
The \emph{heart} of $(\sX,\sY)$,  $\sH = \sX \cap \Sigma \sY$,  is an abelian category \cite{BBD}; 
\item \emph{co-t-structure} \cite{Pauk} (or \emph{weight structure} \cite{Bondarko10}) if in addition $\Sigma^{-1} \sX \subseteq \sX$ (equivalently, $\Sigma \sY \subseteq \sY)$.
Its coheart $\sS = \Sigma \sX \cap \sY$ is a \emph{presilting subcategory}, i.e. $\Hom_\sD(\sS,\Sigma^{>0} \sS) = 0$, see, e.g. \cite{AI}. 
\end{itemize}
\end{definition}

A (co-)t-structure is \emph{bounded} if $\sD = \bigcup_{i \in \bZ} \Sigma^i \sX = \bigcup_{i \in \bZ} \Sigma^i \sY$.
A co-t-structure is bounded if and only if its coheart $\sS$ is a \emph{silting subcategory}, i.e. $\sS$ is presilting and the thick subcategory of $\sD$ generated by $\sS$ is $\sD$ \cite[Theorem 4.20]{MSSS}.

A co-t-structure $(\sU,\sV)$ is said to be \emph{left adjacent} to a t-structure $(\sX,\sY)$ if $\sV = \sX$ \cite{Bondarko10}. Analogously, $(\sU,\sV)$ is \emph{right adjacent} to $(\sX,\sY)$ if $\sU = \sY$.

Note that for a t-structure $(\sX,\sY)$, the inclusion $\sX \to \sD$ has a right adjoint and the inclusion $\sY \to \sD$ has a left adjoint; these are given by the truncation functors. Therefore, the subcategory $\sX$ is always contravariantly finite in $\sD$ and the subcategory $\sY$ is always covariantly finite in $\sD$, and the approximations are even functorial.
In contrast, if $(\sU,\sV)$ is a co-t-structure in $\sD$ then $\sU$ is contravariantly finite in $\sD$ and $\sV$ is covariantly finite in $\sD$, but these approximations need not be functorial.

We recall the following standard characterisation of bounded t-structures; see, for example, \cite[Lemma 3.2]{Bridgeland}.

\begin{lemma}
Let $(\sX,\sY)$ be a t-structure in $\sD$ with heart $\sH$. The following conditions are equivalent:
\begin{enumerate}
\item the t-structure $(\sX,\sY)$ is bounded;
\item $\sX = \bigcup_{n \geq 0} \Sigma^n \sH * \Sigma^{n-1} \sH * \cdots * \sH$ and $\sY = \bigcup_{n <0} \Sigma^{-1} \sH * \Sigma^{-2} \sH * \cdots * \Sigma^n \sH$;
\item $\sD = \bigcup_{n \geq m} \Sigma^n \sH * \Sigma^{n-1} \sH * \cdots * \Sigma^{m+1} \sH * \Sigma^m \sH$.
\end{enumerate}
\end{lemma}

\begin{definition}
A bounded t-structure $(\sX,\sY)$ in $\sD$ with heart $\sH$ is called \emph{algebraic} if $\sH$ is a length category (every object has finite length, i.e.\ is Artinian as well as Noetherian) with only finitely many simple objects.
\end{definition} 

For example, the property of being algebraic holds for bounded t-structures whose hearts are module categories over finite-dimensional algebras.

\subsection{Hereditary algebras and negative cluster categories} \label{sec:hereditary}
For this section, $\sD = \Db(\kk Q)$ for some finite acyclic quiver $Q$. The main reference for the structure of derived categories of hereditary algebras (equivalently, path algebras of acyclic quivers) is \cite{Happel}. 

Recall that an algebra $A$ is \emph{hereditary} if it is of global dimension $0$ or $1$, i.e.\ if the bifunctors $\Ext^n_A(-,-)$ are zero for $n \geq 2$.
Typical examples are the path algebras $A=\kk Q$.
A well-known lemma says that each object of $\Db(\kk Q)$ decomposes as a direct sum of its cohomology. In particular, its Auslander--Reiten (AR) quiver has the following form:

\begin{center}
\begin{tikzpicture}

\draw (-6,2) -- (6,2);
\draw (-6,0) -- (6,0);

\draw (-4.5,0) -- (-4.5,2);
\draw (-1.5,0) -- (-1.5,2);
\draw (1.5,0) -- (1.5,2);
\draw (4.5,0) -- (4.5,2);

\node at (-5,1) {$\cdots$};
\node at (-3,1) {$\Sigma^{-1} \mod{\kk Q}$};
\node at (0,1) {$\mod{\kk Q}$};
\node at (3,1) {$\Sigma \mod{\kk Q}$};
\node at (5,1) {$\cdots$};

\end{tikzpicture}
\end{center}

Morphisms go from left to right, and since for two modules $M, N \in \mod{\kk Q}$, one has $\Hom_{\Db(\kk Q)}(M,\Sigma^n N) = \Ext^n_{\kk Q}(M,N)$, nonzero morphisms exist only from one degree to the next and not any higher.

In this article we will be interested in a certain orbit category of $\Db(\kk Q)$ which is constructed as follows; see \cite{BMRRT}. Let $w \geq 1$ and $F = \Sigma^w \SSS$. We define the category
\[
\Cw = \Cw(\kk Q) \coloneqq \Db(\kk Q)/F,
\]
with the same objects as $\sD$ and whose morphisms are given by
\[
\Hom_{\Cw}(x,y) = \bigoplus_{n \in \bZ} \Hom_{\Db(\kk Q)}(x, F^n y).
\]
We write $\pi \colon \Db(\kk Q) \to \Cw$ for the natural projection functor.
The orbit category $\Cw$ is a triangulated category (see \cite{Keller}) which is $(-w)$-Calabi-Yau; this uses that the algebra $\kk Q$ is hereditary.
Compare $\Cw$ with the construction of the \emph{$m$-cluster category}, $\Db(\kk Q)/\Sigma^{-m}\SSS$, for $m \geq 2$, which is an $m$-Calabi-Yau triangulated category. 
As such we will refer to $\Cw$ as a \emph{negative cluster category} or a \emph{$(-w)$-cluster category}.

It is often convenient to compute inside $\Cw$ using a so-called fundamental domain in $\Db(\kk Q)$. Let $(\sX,\sY)$ be the standard t-structure in $\Db(\kk Q)$, i.e.\ the t-structure with $\sX = \sD^{\geq0}(\kk Q), \sY = \sD^{<0}(\kk Q)$ and heart $\sH = \mod{\kk Q}$. We define the \emph{fundamental domain} of $\Cw$ to be
\[
\Fw \coloneqq \sX \cap \Sigma^w \SSS \sY.
\]
The natural projection functor $\pi \colon \Db(\kk Q) \to \Cw$ induces a bijection
\[
\{\text{indecomposable objects in } \Fw\} \bij \{\text{indecomposable objects of } \Cw\}.
\]
Below we give a schematic of the fundamental domain $\Fw$ inside $\Db(\kk Q)$,

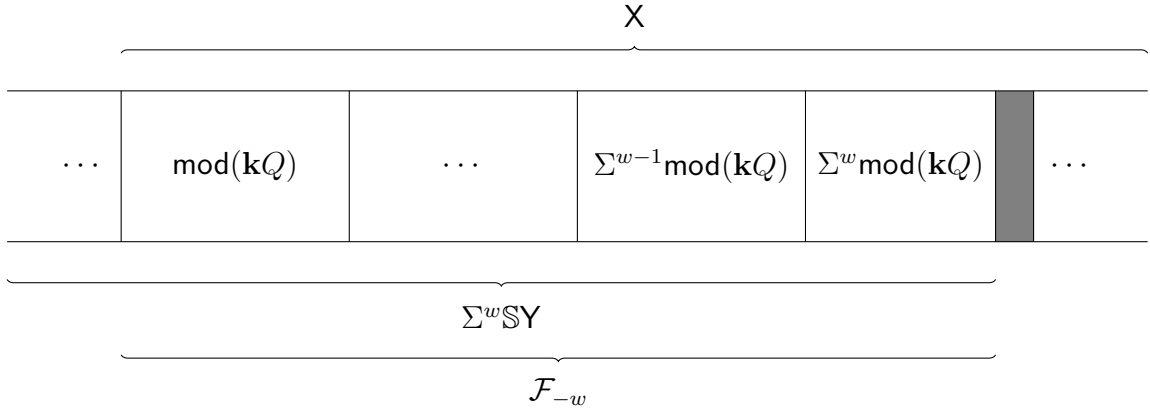
\begin{center}
\begin{tikzpicture}

\fill[gray] (7,0) -- (7.5,0) -- (7.5,2) -- (7,2) -- cycle;

\draw[decoration={brace},decorate] (-4.5,2.5) --  (9,2.5);
\draw[decoration={brace,mirror},decorate] (-6,-0.5) -- (7,-0.5);
\draw[decoration={brace,mirror},decorate] (-4.5,-1.5) -- (7,-1.5);

\draw (-6,2) -- (9,2);
\draw (-6,0) -- (9,0);

\draw (-4.5,0) -- (-4.5,2);
\draw (-1.5,0) -- (-1.5,2);
\draw (1.5,0) -- (1.5,2);
\draw (4.5,0) -- (4.5,2);
\draw (7,0) -- (7,2);
\draw (7.5,0) -- (7.5,2);

\node at (-5,1) {$\cdots$};
\node at (-3,1) {$\mod{\kk Q}$};
\node at (0,1) {$\cdots$};
\node at (3,1) {$\Sigma^{w-1} \mod{\kk Q}$};
\node at (5.75,1) {$\Sigma^w \mod{\kk Q}$};
\node at (8,1) {$\cdots$};
\node at (2.25,3) {$\sX$}; 
\node at (0.5,-1) {$\Sigma^w \SSS \sY$};
\node at (1.25,-2) {$\Fw$};

\end{tikzpicture}
\end{center}

\noindent
where the shaded grey region on the right-hand side comprises $\Sigma^w \inj{\kk Q}$, where $\inj{\kk Q}$ denotes the full subcategory of injective $\kk Q$-modules.

Finally, working inside $\Fw$ allows us to compute Hom spaces easily.

\begin{lemma}[{\cite[Lemma 3.4]{IJ}}] \label{lem:morphisms}
Suppose $x,y \in \Fw$ and $0 \leq i \leq w$. Then
\[
\Hom_{\Cw}(x,\Sigma^{-i} y) = \Hom_{\Db(\kk Q)}(x,\Sigma^{-i}y) \oplus D\Hom_{\Db(\kk Q)}(y, \Sigma^{i-w} x).
\]
\end{lemma}

\section{Functorially finite hearts} \label{sec:funct-finite}

In this section we establish an unexpected characterisation of algebraic t-structures: our third main theorem, which we state and prove first, relates homological properties of hearts to approximation properties. 
The theorem extends a characterisation of Bondarko in terms of co-t-structures; see \cite[Theorem 5.3.1]{Bondarko19}.

In this section, $\sD$ will be a Hom-finite, Krull-Schmidt triangulated category. We start by giving two definitions following \cite{AST} and \cite{NSZ}.

\begin{definition}
Let $(\sX,\sY)$ be a t-structure in $\sD$. The \emph{projective coheart} of $(\sX,\sY)$ is $\sS = {}\orth\Sigma \sX \cap \sX$ and the \emph{injective coheart} of $(\sX,\sY)$ is $\sC = \sY\orth \cap \Sigma \sY$.
\end{definition}

As in \cite{NSZ}, we do not require the existence of a co-t-structure adjacent to $(\sX,\sY)$ in order to consider the projective or injective cohearts.
The projective and injective cohearts were also considered in \cite{AST}, where the objects of the projective and injective cohearts are called the `Ext-projectives of the aisle $\sX$' and the `Ext-injectives of the coaisle $\Sigma \sY$', respectively. This terminology inspires the next definition.

\begin{definition}
Let $(\sX,\sY)$ be a t-structure in $\sD$ with heart $\sH = \sX \cap \Sigma \sY$. We say that $(\sX,\sY)$ has \emph{enough Ext-projectives} if $\sS$ is contravariantly finite in $\sX$ and $\sS\orth \cap \sH = 0$.
Similarly, we say $(\sX,\sY)$ has \emph{enough Ext-injectives} if $\sC$ is covariantly finite in $\Sigma \sY$ and ${}\orth \sC \cap \sH = 0$.
\end{definition}

\begin{remark}
The projective coheart $\sS$ and injective coheart $\sC$ of a t-structure $(\sX,\sY)$ satisfy $\Hom(\sS, \Sigma^{>0} \sS) = 0$ and $\Hom(\sC, \Sigma^{>0} \sC) = 0$; that is, they are examples of \emph{presilting subcategories}.
In most examples we have in mind, the presilting subcategories have an additive generator and therefore are automatically functorially finite, so the hypotheses $\sS$ is contravariantly finite in $\sD$ and $\sC$ is covariantly finite in $\sD$ are quite mild.
\end{remark}

\begin{theorem} \label{thm:heart-finite}
Let $\sD$ be a Hom-finite, Krull-Schmidt, $\kk$-linear triangulated category. Let $(\sX,\sY)$ be a bounded t-structure in $\sD$ with heart $\sH = \sX \cap \Sigma \sY$, projective coheart $\sS = {}\orth\Sigma \sX \cap \sX$ and injective coheart $\sC = \sY\orth \cap \Sigma \sY$.
The following are equivalent:
\begin{enumerate}
\item $\sH$ is contravariantly finite (resp.\ covariantly finite) in $\sD$ and $(\sX,\sY)$ has enough Ext-injectives (resp.\ Ext-projectives).
\item $\sH$ has enough injectives (resp.\ projectives) and each injective object $e \in \sH$ admits an $\sH$-monomorphism $e \into H^0(c)$ for some $c \in \sC$ (resp.\ each projective object $p \in \sH$ admits an $\sH$-epimorphism $H^0(s) \onto p$ for some $s \in \sS$).
\item $(\sX,\sY)$ has a right (resp.\ left) adjacent co-t-structure.
\end{enumerate}
\end{theorem}

The characterisation $(2) \Longleftrightarrow (3)$ was observed in \cite[Theorem 5.3.1]{Bondarko19}; we provide details of the argument for $(2) \Longrightarrow (3)$ for the convenience of the reader.
We only prove the unbracketed statements; the bracketed statements are dual.

\begin{proof}
First we prove $(1) \Longrightarrow (2)$. Suppose $\sH$ is contravariantly finite in $\sD$. We must show that every object of $\sH$ admits an injective envelope. Let $h \in \sH$ and take a minimal right $\sH$-approximation of $\Sigma h$:
\[
h' \to \Sigma h \to \Sigma e \to \Sigma h'.
\]
By the triangulated Wakamatsu lemma $\Sigma e \in \sH\orth$ (see e.g.\ \cite[Lemma 2.1]{Jorgensen}; this is where the hypotheses that $\sD$ is Hom-finite and Krull-Schmidt are used). Rotating this triangle gives us the triangle $h \to e \to h' \to \Sigma h$, whence extension closure of $\sH$ means that $e \in \sH$. Thus, $e \in \sH \cap (\Sigma^{-1} \sH)\orth$, i.e.\ $\Ext^1_\sH(-,e) = 0$ and $e$ is injective. However, each short exact sequence $0 \to h_1 \to h_2 \to h_3 \to 0$ in $\sH$ corresponds to a distinguished triangle $h_1 \to h_2 \to h_3 \to \Sigma h_1$ and vice versa so that we have a short exact sequence $0 \to h \to e \to h' \to 0$, in which case $h \into e$ is an injective envelope. 

Now suppose $e \neq 0$ is an injective object of $\sH$. Since $(\sX,\sY)$ has enough Ext-injectives, we can take a (minimal) left $\sC$-approximation of $e$ and extend it to a distinguished triangle
\[
z \to e \to c \to \Sigma z.
\]
We claim that $z \in \sY$. Since $e \neq 0$ and ${}\orth \sC \cap \sH = 0$, we have that $c \neq 0$ and the Wakamatsu lemma tells us that $z \in {}\orth \sC$.
We also see that $z \in \Sigma \sY$ because $e \in \sX \cap \Sigma \sY$, $\Sigma^{-1} c \in \Sigma^{-1} \sC  \subseteq \sY \subseteq \Sigma \sY$ and $\Sigma \sY$ is extension closed.
Truncating $z$ with respect to $(\sX,\sY)$ gives a triangle,
\[
h_z \to z \to y_z \to \Sigma h_z, 
\]
with $h_z \in \sH$. Applying $\Hom(-,\sC)$ to this triangle and using $\Hom(\Sigma^{-1}\sY,\sC)\subseteq\Hom(\sY,\sC)=0$ reveals that $h_z \in {}\orth \sC \cap \sH = 0$. Hence $z \simeq y_z \in \sY$, as claimed.

Finally, consider the truncation triangle of $c$ with respect to $(\sX,\sY)$:
\[
h_c \to c \to y_c \to \Sigma h_c,
\]
noting that $h_c \in \sH$, that is $H^0(c) = h_c$. Since $e \in \sX$, the composition $e \to c \to y_c$ is zero, which means we obtain the following octahedral diagram.
\[
\xymatrix@!R=8px{
                                                    & e \ar@{=}[r] \ar[d]    & e \ar[d]                       &  \\
\Sigma^{-1} y_c \ar[r] \ar@{=}[d] & h_c \ar[r] \ar[d]         & c \ar[r] \ar[d]              & y_c \ar@{=}[d] \\
\Sigma^{-1} y_c \ar[r]                  & h \ar[r] \ar[d]             & \Sigma z \ar[r] \ar[d]  & y_c \\
                                                   & \Sigma e \ar@{=}[r]  & \Sigma e                    &
}
\]

From the left-hand vertical triangle we read off that $h \in \sX$ and from the lower horizontal triangle we read off that $h \in \Sigma \sY$, using $z\in \sY$. Hence, $h \in \sH$, meaning that the triangle $e \to h_c \to h \to \Sigma e$ corresponds to a short exact sequence $0 \to e \to h_c \to h \to 0$ in $\sH$. Hence, the injective object $e \in \sH$ admits an $\sH$-monomorphism $e \into H^0(c)$ for some $c \in \sC$. 

$(2) \Longrightarrow (3)$. 
Let $\sU = \sY$ and $\sV = \sY\orth$, we need to show that $(\sU,\sV)$ is a co-t-structure in $\sD$. 
Hom-vanishing $\Hom(\sU, \sV) = 0$ is immediate; closure under direct summands and closure under shifts are clear because $\sU = \sY$ is the co-aisle of a t-structure. 
Hence, we only need to exhibit for each $d \in \sD$ the approximation triangle $u \to d \to v \to \Sigma u$ with $u \in \sU$ and $v \in \sV$.
Since $(\sX,\sY)$ is bounded, for each $d \in \sD$ there exists $n \in \bZ$ such that $d \in \Sigma^n \sY = \Sigma^n \sU$. Hence, if $n \leq 0$, we have $d \in \sY$ and the trivial triangle $d \rightlabel{1} d \too 0 \too \Sigma d$ suffices. Thus we assume $n > 0$ and proceed by induction on $n$.

Suppose $n = 1$, i.e.\ $d \in \Sigma \sY$ and consider its $(\sX,\sY)$-truncation triangle, 
\[
h_d \to d \to y_d \to \Sigma h_d,
\] 
noting that $h_d \in \sX \cap \Sigma \sY = \sH$. 
Since $\sH$ has enough injectives, there exists an injective envelope $h_d \into e$, i.e.\ $e \in \sH$ injective. By hypothesis, there exists an $\sH$-monomorphism $e \into h_c$, where $h_c = H^0(c)$ for some $c \in \sC$, the injective coheart. Let $h'' \in \sH$ be the cokernel of the composition of these two morphisms:
$0 \to h_d \to h_c \to h'' \to 0$.

Since $\Hom(\Sigma^{-1} y_d, c) = 0$ due to $\Sigma^{-1} \sY \subseteq \sY$ and $\sC \subseteq \sY\orth$, the morphism $h_d \to h_c$ gives rise to a commutative diagram of triangles,
\[
\xymatrix@!R=8px{
\Sigma^{-1} y_d \ar[r] \ar[d] & h_d \ar[r] \ar[d] & d \ar[r] \ar[d] & y_d \ar[d]  \\
\Sigma^{-1} y_c \ar[r]          & h_c \ar[r]           & c \ar[r]          & y_c
}
\]
which, by \cite[Proposition 1.1.11]{BBD}, extends to the $3 \times 3$ diagram below.
\[
\xymatrix@!R=8px{
\Sigma^{-1} y \ar[r] \ar[d]    & \Sigma^{-1} h'' \ar[r] \ar[d]  & u \ar[r] \ar[d]   & y \ar[d] \\  
\Sigma^{-1} y_d \ar[r] \ar[d] & h_d \ar[r] \ar[d]                  & d \ar[r] \ar[d]   & y_d \ar[d]  \\
\Sigma^{-1} y_c \ar[r] \ar[d] & h_c \ar[r] \ar[d]                  & c \ar[r] \ar[d]    & y_c \ar[d] \\
y \ar[r]                                  & h'' \ar[r]                             & \Sigma u \ar[r] & \Sigma y
}
\]
From the right-hand vertical triangle we read off $y \in \sY$ and the top-most horizontal triangle gives $u \in \sU = \sY$, using $\Sigma^{-1} h'' \in \Sigma^{-1} \sH \subseteq \sY$. Since $c \in \sC = \sY\orth \cap \Sigma \sY = \sV \cap \Sigma \sU \subseteq \sV$, we see that $u \to d \to c \to \Sigma u$ is a $(\sU,\sV)$-approximation triangle of $d$.

Now suppose $d \in \Sigma^n \sY$ for some $n > 1$. By induction, there is a $(\sU,\sV)$-approximation triangle $u_1 \to \Sigma^{-1} d \to v_1 \to \Sigma u_1$ with $u_1 \in \sU$ and $v_1 \in \sV$. 
By the base step of the induction, there is a $(\sU,\sV)$-approximation triangle $u \to \Sigma u_1 \to c \to \Sigma u$ with $u \in \sU$ and $c \in \sC$.
Applying the octahedral axiom, we obtain the following commutative diagram,
\[
\xymatrix@!R=8px{
                             & v_1 \ar@{=}[r] \ar[d]      & v_1 \ar[d]      &  \\
u \ar[r] \ar@{=}[d] & \Sigma u_1 \ar[r] \ar[d]  & c \ar[r] \ar[d]  & \Sigma u \ar@{=}[d] \\
u \ar[r]                   & d \ar[r] \ar[d]                  & v \ar[r] \ar[d]  & \Sigma u \\
                             & \Sigma v_1 \ar@{=}[r]    & \Sigma v_1   &
}
\]
in which $v \in \sV$ because $c \in \sC \subseteq \sV$ and $\Sigma \sV \subseteq \sV$. Hence, $u \to d \to v \to \Sigma u$ provides the required $(\sU,\sV)$-approximation triangle of $d$. 
Hence, the t-structure $(\sX,\sY)$ has a right adjacent co-t-structure.

$(3) \Longrightarrow (1)$. 
Since there is a right adjacent co-t-structure, it follows that $\Sigma \sY$ is contravariantly finite in $\sD$. Thus, to obtain that $\sH$ is contravariantly finite in $\sD$ it is enough to show that $\sH$ is contravariantly finite in $\Sigma \sY$. This is straightforward: let $d \in \Sigma \sY$ and take the truncation triangle with respect to $(\sX,\sY)$, 
\[
x \to d \to y \to \Sigma x.
\]
Here, the morphism $x \to d$ is a right $\sX$-approximation, so in particular, any morphism $h \to d$ with $h \in \sH$ factors through $x \to d$. Since $d \in \Sigma \sY$, it follows that $x \in \Sigma \sY$ because $\Sigma^{-1}y \in \Sigma^{-1}\sY \subset \Sigma \sY$. Hence $x \in \sX \cap \Sigma \sY = \sH$, so, in particular, $x \to d$ is a right $\sH$-approximation of $d$. Hence $\sH$ is contravariantly finite in $\Sigma \sY$ and therefore also in $\sD$.

To see that $(\sX,\sY)$ has enough Ext-injectives, let $(\sU,\sV)$ be the co-t-structure right adjacent to the t-structure $(\sX,\sY)$, i.e.\ with $\sU = \sY$, and note that $\sC = \Sigma \sU \cap \sV$ is the coheart of this co-t-structure. We first claim that $\sC$ is covariantly finite in $\Sigma \sY = \Sigma \sU$.
Let $y \in \sY$ and consider a $(\sU,\sV)$-approximation triangle of $\Sigma y$,
$u \to \Sigma y \to c \to \Sigma u$.
Extension closure of $\sU$ and $\sV$ shows that $c \in \Sigma \sU \cap \sV = \sC$. Applying the functor $\Hom(-,\sC)$ to this triangle reveals that the morphism $\Sigma y \to c$ is a left $\sC$-approximation. Hence, $\sC$ is covariantly finite in $\Sigma \sY$.

Finally, we need to show that ${}\orth \sC \cap \sH = 0$. Suppose $0 \neq h \in {}\orth \sC \cap \sH$ and again consider a $(\sU,\sV)$-approximation triangle as above: $u \to h \to c \to \Sigma u$. The argument above showed that $c \in \sC$, whence the morphism $h \to c$ is zero, making $h$ a direct summand of $u \in \sU = \sY$. Hence, $h \in \sX \cap \sY = 0$, as required and $(\sX,\sY)$ has enough Ext-injectives.
\end{proof}

\begin{remark}
In the statement of Theorem~\ref{thm:heart-finite}, the full strength of the condition that the t-structure $(\sX,\sY)$ is bounded is not strictly required. 
The statements involving enough (Ext-)injectives require the t-structure to be bounded above, i.e.\ $\sD = \bigcup_{n\in \bZ} \Sigma^n \sY$; those involving enough (Ext-)projectives require the t-structure to be bounded below, i.e.\ $\sD = \bigcup_{n \in \bZ} \Sigma^n \sX$.
These conditions are only used in the proof of the implication $(2) \Longrightarrow (3)$.
\end{remark}

In the case that $\sD$ is a saturated triangulated category, the statement of Theorem~\ref{thm:heart-finite} becomes simpler and we can omit all mention of the existence of enough Ext-projectives or enough Ext-injectives and the projective and injective cohearts. We recall the definition of a saturated triangulated category from \cite{Bondal-Kapranov,Bondal-vdB}.

\begin{definition}
Let $\sD$ be a triangulated category and $F\colon \sD \to \Mod{\kk}$ a cohomological functor.
 
\begin{enumerate}
\item $F$ is a {\it functor of finite type} if for any $d \in \sD$, $F(\Sigma^i d) \neq 0$ only for finitely many $i$. 
\item If $F$ is contravariant, then it is called {\it representable} if there is a functor isomorphism between $F$ and $\Hom_\sD (-,x)$ for some $x \in \sD$. The representing object $x$ is unique, up to isomorphism. Dually for covariant $F$.
\end{enumerate} 
The triangulated category $\sD$ is said to be
\begin{enumerate}[resume]
\item {\it of finite type} if for every pair of objects $x, y \in \sD$, the space $\Ext^i_\sD (x,y)$ is finite dimensional and it vanishes for almost all $i$. 
\item {\it saturated} if it is of finite type, and if every cohomological functor of finite type is representable. 
\end{enumerate}
\end{definition}

\begin{example}
Examples of saturated triangulated categories include the bounded derived categories $\Db(A)$ of finite-dimensional $\kk$-algebras $A$ of finite global dimension \cite[Theorem 2.11]{Bondal-Kapranov}, the bounded derived categories $\Db(\coh{X})$ of coherent sheaves on a smooth projective variety $X$ \cite[Theorem 2.14]{Bondal-Kapranov} (see also \cite[Theorem A.1]{Bondal-vdB} for a more general statement), and more generally, Ext-finite triangulated categories with a strong generator \cite[Theorem 1.3]{Bondal-vdB}.
\end{example}

We now state the simplification of Theorem~\ref{thm:heart-finite} for the saturated case.

\begin{corollary}[Theorem~\ref{thm:functorially-finite}] \label{cor:saturated}
Let $\sD$ be a Hom-finite, Krull-Schmidt, saturated triangulated category. Let $(\sX,\sY)$ be a bounded t-structure in $\sD$ with heart $\sH$. The following are equivalent:
\begin{enumerate}
\item $\sH$ is contravariantly finite (resp.\ covariantly finite) in $\sD$.
\item $\sH$ has enough injectives (resp.\ projectives).
\item $(\sX,\sY)$ has a right (resp.\ left) adjacent co-t-structure, i.e.\ there is a co-t-structure $(\sY, \sY^\perp)$ (resp.\ $(\,^\perp \sX, \sX)$).   
\end{enumerate}
\end{corollary}

\begin{proof}
Again we show only the unbracketed statements; the bracketed statements are dual.
The implications $(1) \Longrightarrow (2)$ and $(3) \Longrightarrow (1)$ are contained in the proof of Theorem~\ref{thm:heart-finite}. 
The implication $(2) \Longrightarrow (3)$ is due to \cite[Theorem 5.3.1(IV)]{Bondarko19}; we include details for the convenience of the reader.

By the dual of \cite[Lemma 1.3]{AST} (see also \cite[Lemma 2]{NSZ}), the functor $H^0(-)|_\sC \colon \sC \to \inj{\sH}$ is fully faithful, where $\sC = \sY\orth \cap \Sigma \sY$ is the injective coheart of $(\sX,\sY)$. 
Moreover, by considering suitable truncation triangles, for each $c \in \sC$ there is a natural isomorphism $\Hom_\sD(-,c) \simeq \Hom_\sH(H^0(-),H^0(c))$.
We claim that $H^0(-)|_\sC$ is an equivalence of categories, that is
for each $e \in \inj{\sH}$, we need to show that there is an isomorphism $H^0(c) \simeq e$ for some $c \in \sC$.

Consider the functor $\Hom_\sH(H^0(-),e) \colon \sD \to \mod{\kk}$. By saturatedness, $\Hom_\sH(H^0(-),e)$ is a cohomological functor of finite type, and hence is representable, i.e.\ there exists an object $c \in \sD$ such that $\Hom_\sH(H^0(-),e) \simeq \Hom_\sD(-,c)$. We note that $c \in \sC$ since $H^0(\Sigma x) = 0$ for $x \in \sX$ and $H^0(y) = 0$ for $y \in \sY$, which together imply that $\Hom_\sD(\Sigma \sX,c) = 0$, putting $c \in (\Sigma \sX)\orth = \Sigma \sY$, and $\Hom_\sD(\sY,c) = 0$, putting $c \in \sY\orth$.

Combining the natural isomorphisms in the two preceding paragraphs and restricting to $\sH$ gives rise to a natural isomorphism
\[
\theta \colon \Hom_\sH(-, H^0(c)) \to \Hom_\sH(-,e).
\]
Let $\alpha = \theta_{H^0(c)}(1_{H^0(c)})$ and $\beta = \theta_e^{-1}(1_e)$. Applying the natural isomorphism to the morphism $\alpha \colon H^0(c) \to e$ gives a commutative diagram:
\[
\xymatrix{
\Hom_\sH(e,H^0(c)) \ar[rr]^-{\theta_e} \ar[d]_-{\Hom_\sH(\alpha, H^0(c))} & & \Hom_\sH(e,e) \ar[d]^-{\Hom_\sH(\alpha,e)} \\
\Hom_\sH(H^0(c),H^0(c)) \ar[rr]_-{\theta_{H^0(c)}}                                     & & \Hom_\sH(H^0(c),e)
}
\]
Chasing $\beta \colon e \to H^0(c)$ through this diagram shows that $\beta \alpha = 1_{H^0(c)}$. Similarly, chasing $\alpha$ through the corresponding diagram constructed using $\beta$ shows $\alpha \beta = 1_e$. 
Hence $e \simeq H^0(c)$ and $H^0(-)|_\sC \colon \sC \to \inj{\sH}$ is an equivalence of categories.
Note, in particular, that $\beta \colon e \to H^0(c)$ provides the required $\sH$-monomorphism so that we can now apply the argument $(2) \implies (3)$ in the proof of Theorem~\ref{thm:heart-finite}.
\end{proof}

\begin{remark}
The argument in the proof of Corollary~\ref{cor:saturated} also shows that the condition that the t-structure $(\sX,\sY)$ has enough Ext-injectives (resp., Ext-projectives) is implicit in the saturated case since condition $(3)$ remains unchanged in both statements.
\end{remark}

Using Corollary~\ref{cor:saturated} we are able to establish the existence of Auslander--Reiten sequences in hearts of bounded t-structures inside saturated triangulated categories.

\begin{corollary} \label{cor:AR}
Suppose $\sD$ is a Hom-finite, Krull-Schmidt, saturated triangulated category. Let $\sH$ be the heart of a bounded t-structure in $\sD$.
\begin{enumerate}[label=(\arabic*)]
\item If $\sH$ has enough injectives then any indecomposable object of $\sH$ which is not Ext-projective is the third object of an Auslander--Reiten sequence.
\item If $\sH$ has enough projectives then any indecomposable object of $\sH$ which is not Ext-injective is the first object of an Auslander--Reiten sequence.
\end{enumerate}
\end{corollary}

\begin{proof}
We only prove the first statement; the second is dual.
Suppose $h \in \ind{\sH}$ is not Ext-projective, i.e.\ $\Ext^1_{\sH}(h,-) = \Hom_{\sD}(h,\Sigma -)|_\sH \neq 0$. Since $\sD$ is saturated it has a Serre functor; see, e.g. \cite{Kawamata}. Therefore, by \cite[Theorem I.2.4]{RvdB}, $\sD$ has Auslander--Reiten triangles. In particular, there is an Auslander--Reiten triangle $x  \to y \to h \to \Sigma x$ in $\sD$. Since $\sH$ has enough injectives, by Corollary~\ref{cor:saturated}, $\sH$ is contravariantly finite in $\sD$. Hence, $x$ admits a minimal right $\sH$-approximation. Thus we can now apply \cite[Theorem 3.1]{Jorgensen} to conclude that there is an Auslander--Reiten sequence $0 \to h' \to h'' \to h \to 0$ in $\sH$.
\end{proof}

Finally, we remark that Theorem~\ref{thm:heart-finite} allows us to recognise module categories of finite-dimensional algebras via the approximation theory of the heart inside an ambient triangulated category. 

\begin{corollary}
Let $\sD$ be a Hom-finite saturated $\kk$-linear triangulated category.
Suppose $(\sX,\sY)$ is an algebraic t-structure in $\sD$ with heart $\sH$. The heart $\sH$ is covariantly finite in $\sD$ if and only if $\sH \simeq \mod A$ for a finite-dimensional $\kk$-algebra $A$.
\end{corollary}

\begin{proof}
By \cite[p.\ 55]{Bass}, $\sH$ has enough projectives if and only if $\sH \simeq \mod A$ for some finite-dimensional $\kk$-algebra $A$. The result now follows from Corollary~\ref{cor:saturated}.
\end{proof}

\begin{corollary} \label{cor:heart-finite}
Let $A$ be a finite-dimensional $\kk$-algebra. If $(\sX,\sY)$ is an algebraic t-structure in $\Db(A)$ with heart $\sH$, then $\sH$ is functorially finite in $\Db(A)$. 
\end{corollary}

\begin{proof}
Let $A$ be a finite-dimensional algebra and suppose $(\sX,\sY)$ is an algebraic t-structure in $\Db(A)$. We show that $\sH$ is covariantly finite in $\sD$; dually one can show $\sH$ is contravariantly finite. By the K\"onig-Yang correspondences \cite{KY} and \cite[Proposition 2.20]{AI}, the projective coheart $\sS = {}\orth \Sigma \sX \cap \sX$ is a silting subcategory of $\Kb(\proj{A})$ such that $\sS = \add{s}$ for some object $s \in \Kb(\proj{A})$. In particular, $\sS$ is a functorially finite subcategory of $\Db(A)$.
Furthermore, $(\sX,\sY) = ((\Sigma^{<0} \sS)\orth, (\Sigma^{\geq 0} \sS)\orth)$ by \cite[Theorem 6.1]{KY}.
Now, applying \cite[Proposition 3.2]{IYa}, we observe that $(\cosusp{\Sigma^{-1} \sS}, (\Sigma^{<0} \sS)\orth)$ is a co-t-structure in $\Db(A)$ which is left adjacent to the t-structure $(\sX,\sY)$. Hence, by Theorem~\ref{thm:heart-finite}, $\sH$ is covariantly finite in $\Db(A)$.
\end{proof}

The following example, however, shows that there can be algebraic hearts inside a Hom-finite Krull-Schmidt triangulated category which are not necessarily functorially finite.

\begin{example} \label{ex:not-funct-finite}
Let $\sH$ be a standard stable homogeneous tube. This is a length category in which every object is uniserial. Moreover, $\sH$ contains a single simple object, but no injective or projective objects except the zero object. It has infinitely many indecomposable objects.
The category $\sD = \Db(\sH)$ is a $1$-Calabi-Yau triangulated category in which the only torsion pairs are trivial or shifts of the standard t-structure, which is bounded, see \cite[Theorem 9.1]{CSP16}. Since $\sH$ is length with one isoclass of simple objects, it is algebraic; but it doesn't have enough projectives or injectives and is therefore neither contravariantly finite nor covariantly finite in $\sD$. 
\end{example}

The next example gives a typical application of the theorem where the intrinsic property (presence or absence of enough injective and projective objects) is used to make a statement about finiteness.

\begin{example}
Let $\sH = \mod{\kk\tA{1}}$ be the category of finite-dimensional representations of the Kronecker quiver. In the derived category $\sD = \Db(\sH)$, this heart is obviously algebraic. As is well known, $\sD$ is equivalent to the bounded derived category of coherent sheaves on the projective line $\bP^1$ over $\kk$, giving rise to another heart $\sH' = \coh{\bP^1}$ in $\sD$. The abelian category $\sH'$ has neither injective nor projective objects apart from 0. Hence by Theorem~\ref{thm:heart-finite}, the heart $\sH'$ is neither covariantly nor contravariantly finite in $\sD$. 
\end{example}

\begin{remark}
It would be interesting to investigate when Corollary~\ref{cor:saturated} holds without the saturatedness assumption on the triangulated category.

As an example, consider $\sD = \Db(\Qcoh{\bP^1})$, the bounded derived category of quasi-coherent sheaves on $\bP^1$. The standard heart $\sH = \Qcoh{\bP^1}$ is a hereditary abelian category with enough injective objects but no nonzero projective objects.
The hereditary property, together with the fact that $\sH$ is the heart of a split t-structure in $\sD$, 
means the injective coheart $\sC = \sY\orth \cap \Sigma \sY = (\Sigma^{-1} \sH)\orth \cap \sH = \inj{\sH}$.
Hence condition $(2)$ in Theorem~\ref{thm:heart-finite} simplifies to the condition that $\sH$ has enough injectives. By Theorem~\ref{thm:heart-finite}, we can conclude that $\sH$ is contravariantly finite in $\sD$.
We can also see this explicitly: every object in $\sD$ splits into a finite direct sum of its cohomology sheaves. Let $d$ be an object of $\sD$ which we can assume to be $d = \Sigma^i A$ with $A \in \sH$ and $i \in \bZ$. Nonzero right $\sH$-approximations $H \to \Sigma^i A$ can only exist if $i=0$ or $i=1$. If $i=0$, they are trivial (take $H=A$). If $i=1$, let $I$ be an injective hull of $A$, leading to a short exact sequence $0 \to A \to I \to H \to 0$, i.e.\ $H$ is the first co-szyzgy of $A$. 
The resulting morphism $H \to \Sigma A$ is a right $\sH$-approximation. 

We observe that the lack of projective objects means that $\sH$ is not an enveloping subcategory of $\sD$.
As above, the hereditary property means that it is sufficient to consider nontrivial minimal left $\sH$-approximations only for objects of $\Sigma^{-1} \sH$. Let $A \in \sH$, take a minimal left $\sH$-approximation of $\Sigma^{-1} A$ and extend it to a distinguished triangle $\Sigma^{-1} H \to \Sigma^{-1} P \to \Sigma^{-1} A \to H$. 
Since $A, H \in \sH$, we have $P \in \sH$.
Because the approximation is minimal, the Wakamatsu lemma tells us that $\Sigma^{-1} P \in {}\orth \sH$. 
This means that $P \in \sH \cap {}\orth (\Sigma \sH) = \proj{\sH} = 0$. Therefore, no such nontrivial minimal left $\sH$-approximation exists and, hence, $\sH$ is not an enveloping subcategory of $\sD$.

We also expect that $\sH$ is not covariantly finite in $\sD$. The above argument fails to apply because $\sD$ is not saturated. This expectation is consistent with Theorem~\ref{thm:heart-finite}: the projective coheart of the standard t-structure is $\sS = {}\orth (\Sigma \sX) \cap \sX = {}\orth (\Sigma \sH) \cap \sH = \proj{\sH} = 0$ and hence the standard heart $\sH = \Qcoh{\bP^1}$ also does not have enough Ext-projectives.
\end{remark}

\section{Orthogonal collections} \label{sec:simple-minded}

In this section we recall the various notions of orthogonal collections and then establish some useful characterisations of them. The main references for the definitions in this section are \cite{CS12,CS15,CS17,CSP20,Dugas,Koenig-Liu}.
From now on, we will assume that the field $\kk$ is algebraically closed.

\begin{definitions}
A collection of objects $\sS$ in $\sD$ is called a \emph{$1$-orthogonal collection} (or simply \emph{orthogonal collection}) if $\dim \Hom_\sD(x,y) = \delta_{xy}$ for every $x,y \in \sS$. Let $w \geq 1$ be an integer. An orthogonal collection $\sS$ is said to be
\begin{enumerate}[label=(\roman*)]
\item (if $w > 1$) \emph{$w$-orthogonal} if $\Hom_\sD(\Sigma^k x,y) = 0$ for $1 \leq k \leq w-1$ and $x,y \in \sS$;
\item \emph{$\infty$-orthogonal} if $\Hom_\sD(\Sigma^k x,y) = 0$ for $k \geq 1$ and $x,y \in \sS$;
\item a \emph{$w$-simple-minded system} if it is $w$-orthogonal and $\sD = \extn{\sS} * \Sigma^{-1} \extn{\sS} * \cdots * \Sigma^{1-w} \extn{\sS}$;
\item a \emph{simple-minded collection} if it is $\infty$-orthogonal and $\sD = \thick{\sD}{\sS}$;
\item a \emph{left $w$-Riedtmann configuration} if it is $w$-orthogonal and $\bigcap_{k=0}^{w-1} (\Sigma^k \sS)\orth = 0$;
\item a \emph{right $w$-Riedtmann configuration} if it is $w$-orthogonal and $\bigcap_{k=0}^{w-1} {}\orth (\Sigma^{-k} \sS) = 0$; 
\item a \emph{$w$-Riedtmann configuration} if it is both a left $w$-Riedtmann configuration and a right $w$-Riedtmann configuration; and,
\item an \emph{$\infty$-Riedtmann configuration} if it is $\infty$-orthogonal and ${}\orth (\Sigma^{< 0} \sS) \cap (\Sigma^{\geq 0} \sS)\orth = 0$. 
\end{enumerate}
\end{definitions}

Our methods work in the case that $\kk$ is not algebraically closed provided one modifies the definition of orthogonal collection to require that $\Hom_\sD(x,x)$ is a division ring for each $x$ in the orthogonal collection. For simplicity of exposition, we choose to work over algebraically closed fields, so that when considering $\Db(\kk Q)$ one can work with quiver representations instead of the more technical representations of species. 

In \cite{Po} an orthogonal collection is called a \emph{system of orthogonal bricks}, in \cite{Dugas} a \emph{set of (pairwise) orthogonal bricks} and in \cite{Asai} a \emph{semibrick}. In \cite{IJ, Jin1,Jin2}  $w$-Riedtmann configurations are called \emph{$(-w)$-Calabi-Yau configurations} in light of \cite[Theorem 6.2]{Jin1} which asserts that if $\sS$ is a $w$-Riedtmann configuration then $\SSS \Sigma^w \sS = \sS$.

\begin{remark}
We have chosen to call collections with vanishing morphisms between distinct objects `$1$-orthogonal collections' rather than `$0$-orthogonal collections' because in $(-1)$-Calabi-Yau categories $\Hom(x,y) \simeq D\Hom(\Sigma y, x)$, which means that $1$-orthogonal collections with vanishing morphisms are the appropriate notion for $(-1)$-CY categories.
As such, one should think of a $w$-orthogonal collection (resp.\ $w$-simple-minded systems, resp.\ $w$-Riedtmann configuration) as being adapted for $(-w)$-CY categories.
\end{remark}

Let $\sX \subseteq \sD$ be a collection of objects in $\sD$. We set $(\sX)_1 = \sX$ and $(\sX)_n = \sX * (\sX)_{n-1}$. We now recall some basic properties of orthogonal collections.

\begin{lemma} \label{lem:basic-properties}
Let $\sS$ be an orthogonal collection in $\sD$. Then the following assertions hold:
\begin{enumerate}[label=(\arabic*)]
\item {\rm (\cite[Lemma 2.7]{Dugas})} $(\sS)_n$ is closed under direct summands for each $n \geq 1$.
\item {\rm (\cite[Lemma 2.3]{Dugas})} $\extn{\sS} = \bigcup_{n \geq 1} (\sS)_n$.
\item {\rm (\cite[Theorem 2.11]{CSP20} \& \cite[Theorem 3.3]{Dugas})} If $\sS \subseteq \sT$ for an orthogonal collection $\sT$ in $\sD$ then $\extn{\sS}$ is functorially finite in $\extn{\sT}$.
\end{enumerate}
\end{lemma}

Lemma~\ref{lem:basic-properties}$(2)$ means that the following definition makes sense.

\begin{definition}[{\cite[Definition 2.5]{Dugas}}]
Let $\sS$ be an orthogonal collection in $\sD$. The \emph{$\sS$-length} (or simply \emph{length}) of $x \in \extn{\sS}$ is the smallest natural number $n$ such that $x \in (\sS)_n$.
\end{definition}

Recall the following characterisation of $w$-simple-minded systems from \cite{CSP20}.

\begin{proposition}[{\cite[Proposition 2.13]{CSP20}}] \label{prop:sms-char}
Let $\sS$ be a collection of indecomposable objects in $\sD$, and let $w \geq 1$ be an integer. The following conditions are equivalent:
\begin{enumerate}[label=(\arabic*)]
\item $\sS$ is a $w$-simple-minded system.
\item $\sS$ is a right $w$-Riedtmann configuration such that $\extn{\sS}$ is covariantly finite in $\sD$.
\item $\sS$ is a left $w$-Riedtmann configuration such that $\extn{\sS}$ is contravariantly finite in $\sD$.
\item $\sS$ is a $w$-Riedtmann configuration such that $\extn{\sS}$ is functorially finite in $\sD$.
\end{enumerate}
\end{proposition}

We next provide an analogue of Proposition~\ref{prop:sms-char} for simple-minded collections.

\begin{proposition} \label{prop:smc-char}
Let $\sS$ be a collection of indecomposable objects in $\sD$. Then $\sS$ is a simple-minded collection if and only if $\sS$ is an $\infty$-Riedtmann configuration such that
$\susp \sS$ is contravariantly finite in $\sD$ and $\cosusp \sS$ is covariantly finite in $\sD$.
\end{proposition}

\begin{proof}
Suppose that $\sS$ is a simple-minded collection in $\sD$. Since both simple-minded collections and $\infty$-Riedtmann configurations are $\infty$-orthogonal collections, we only need to check the other defining properties of $\infty$-Riedtmann configurations. Then, by definition, $\thick{\sD}{\sS} = \sD$ and by, e.g.\ \cite[Lemma 2.7]{CSP20}, we have 
\[
\sD = \bigcup_{n > m} \Sigma^n \extn{\sS} * \Sigma^{n-1} \extn{\sS} * \cdots * \Sigma^{m+1} \extn{\sS} * \Sigma^m \extn{\sS}.
\]
It therefore follows immediately that $(\susp \sS,\cosusp \Sigma^{-1} \sS)$ is a bounded t-structure in $\sD$, and, in particular, $\susp \sS$ is contravariantly finite and $\cosusp \sS$ is covariantly finite. Finally, $(\Sigma^{\geq 0} \sS)\orth = (\susp \sS)\orth = \cosusp \Sigma^{-1} \sS$ and ${}\orth (\Sigma^{< 0} \sS) = {}\orth(\cosusp \Sigma^{-1} \sS) = \susp \sS$, so that ${}\orth (\Sigma^{< 0} \sS) \cap (\Sigma^{\geq 0} \sS)\orth = \susp \sS \cap \cosusp \Sigma^{-1} \sS = 0$. 

Conversely, suppose that $\sS$ is an $\infty$-Riedtmann configuration such that $\susp \sS$ is contravariantly finite in $\sD$ and $\cosusp \sS$ is covariantly finite in $\sD$. To see that $\sS$ is a simple-minded collection we need only to check that $\thick{\sD}{\sS} = \sD$. Let $d$ be an object of $\sD$. Since $\susp \sS$ is contravariantly finite in $\sD$ there is a decomposition triangle
$\tri{x}{d}{y}$ with respect to the t-structure $(\susp \sS, (\Sigma^{\geq 0} \sS)\orth)$. Now, since $\cosusp \sS$ is covariantly finite in $\sD$ there is a decomposition triangle of $y$, $\Sigma^{-1} v \to u \to y \to v$ with respect to the t-structure $({}\orth(\Sigma^{<0} \sS), \cosusp \Sigma^{-1} \sS)$. Note that $\cosusp \Sigma^{-1} \sS \subseteq (\Sigma^{\geq 0} \sS)\orth$ so that $\Sigma^{-1} v \in (\Sigma^{\geq -1} \sS)\orth \subseteq (\Sigma^{\geq 0} \sS)\orth$. Since $(\Sigma^{\geq 0} \sS)\orth$ is a perpendicular category, it is closed under extensions. It follows that $u \in {}\orth(\Sigma^{<0} \sS) \cap (\Sigma^{\geq 0} \sS)\orth$. Hence, since $\sS$ is $\infty$-Riedtmann, we obtain that $u = 0$ and $y \simeq v$. In particular, we get
\[
d \in \susp \sS * \cosusp \Sigma^{-1} \sS = \thick{\sD}{\sS}.
\]
Hence, $\sD = \thick{\sD}{\sS}$ and $\sS$ is a simple-minded collection.
\end{proof}

Finally, we end this section with the following restatement of Corollary~\ref{cor:heart-finite} in the language of simple-minded collections.

\begin{corollary} \label{cor:extn-of-smc-funct-finite}
Suppose $\sD = \Db(A)$ for some finite-dimensional $\kk$-algebra $A$. If $\sS$ is a simple-minded collection in $\sD$, then $\extn{\sS}$ is functorially finite in $\sD$.
\end{corollary}

\section{Simple-minded collections vs $w$-simple minded systems} \label{sec:bijection}

Let $Q$ be a finite acyclic quiver and $w \geq 1$. Recall from Section~\ref{sec:hereditary} the construction of the negative cluster category $\Cw = \Cw(\kk Q)$ as an orbit category of $\sD = \Db(\kk Q)$ and the fundamental domain $\Fw \coloneqq \sX \cap \Sigma^w \SSS \sY$. The aim of this section is to establish the following theorem, which generalises the case of $Q$ simply-laced Dynkin of \cite[Theorem 1.2]{IJ}. The proof of Iyama and Jin uses that subcategories of Hom-finite triangulated categories with finitely many indecomposable objects are automatically functorially finite; we use Theorem~\ref{thm:heart-finite} to obtain functorial finiteness in our setting.

\begin{theorem}[Theorem~\ref{thm:bijection}] \label{thm:SMC-SMS-bijection}
Let $Q$ be an acyclic quiver. The natural projection functor $\pi \colon \Db(\kk Q) \to \Cw$ induces a bijection 
\[ 
\pi \colon
\left\{ \parbox{5.35cm}{\centering simple-minded collections of $\Db(\kk Q)$ contained in $\Fw$} \right\} 
\bij
\left\{ \parbox{4.75cm}{\centering $w$-simple-minded systems in $\Cw$} \right\}. 
\]
\end{theorem}

The proof of this theorem comes in two parts. We must first establish that the map of the theorem induced by the functor $\pi\colon \Db(\kk Q) \to \Cw$, which, by abuse of notation, we also call $\pi$, is well defined and secondly that the map is surjective. Once we know that the map $\pi$ is well defined, injectivity follows immediately because the projection functor $\pi$ induces a bijection between $\ind{\Cw}$ and $\ind{\Fw}$, where $\Fw$ is the fundamental domain; see Section~\ref{sec:hereditary} for notation.

\subsection{The induced map $\pi$ is well defined}

In this section, we establish the following. 

\begin{proposition} \label{prop:SMC-SMS-well-defined}
  Let $Q$ be an acyclic quiver. The natural 
  functor $\pi \colon \Db(\kk Q) \to \Cw$ induces a well-defined map 
\[ 
\pi \colon
\left\{ \parbox{5.35cm}{\centering simple-minded collections of $\Db(\kk Q)$ contained in $\Fw$} \right\} 
\too
\left\{ \parbox{4.75cm}{\centering $w$-simple-minded systems in $\Cw$} \right\}. 
\]
\end{proposition}

In \cite[Proposition 3.6]{IJ}, Iyama and Jin show that $\pi$ induces a well-defined map
\[
\tilde{\pi} \colon
\left\{ \parbox{5.35cm}{\centering simple-minded collections of $\Db(\kk Q)$ contained in $\Fw$} \right\}
\too 
\{w\text{-Riedtmann configurations in $\Cw$}\}.
\]
When $Q$ is simply-laced Dynkin, by Proposition~\ref{prop:sms-char}, this is enough to establish Proposition~\ref{prop:SMC-SMS-well-defined}. However, more work is required for an arbitrary acyclic quiver $Q$. 
In fact, the following example illustrates that $\tilde{\pi}$ is not surjective when $Q$ is not simply-laced Dynkin.

\begin{example}
Let $Q = \tA{1}$ be the Kronecker quiver and $w = 1$. Partition the set of homogeneous tubes into two nonempty, disjoint sets $\Lambda$ and $\Omega$ and set 
\[
\sS \coloneqq \{ S_\lambda \mid \lambda \in \Lambda\} \cup \{ \Sigma S_\omega \mid \omega \in \Omega\},
\]
where $S_\lambda$ (resp.\ $S_\omega$) denotes the quasi-simple modules lying on the mouth of the tubes indexed by $\Lambda$ (resp.\ $\Omega$). Then $\sS$ is a $1$-Riedtmann configuration in $\sC_{-1} \coloneqq \sC_{-1} (\kk \tA{1})$; the required Hom-vanishing needs $\Lambda\neq\varnothing$ and $\Omega\neq\varnothing$. However, $\sS$ is not a simple-minded collection in $\sD \coloneqq \Db(\kk \tA{1})$ since $\thick{\sD}{\sS} \neq \sD$. Note also that $\extn{\sS}_{\sC_{-1}}$ is not functorially finite in $\sC_{-1}$ and so $\sS$ is not a simple-minded system in $\sC_{-1}$.  
\end{example}

To prove Proposition~\ref{prop:SMC-SMS-well-defined}, we need some lemmas. Recall $F = \Sigma^w \SSS\colon \Db(\kk Q) \to \Db(\kk Q)$.

\begin{lemma}\label{lem:smc-vanishing-conditions}
Let $\sS$ be an $\infty$-orthogonal collection of $\Db(\kk Q)$ contained in $\Fw$. 
\begin{enumerate}
\item The set $\{F^n \sS \mid n \in \bZ\}$ is an orthogonal collection in $\Db(\kk Q)$. 
\item For $k \geq 1$, we have $\Ext^1_{\Db(\kk Q)} (\sS, F^k \sS) = 0$. 
\end{enumerate}
\end{lemma}

\begin{proof}
To prove $(1)$, it is enough to show that $\Hom_{\sD} (S_1, F^n S_2) = 0$ for $S_1,  S_2 \in \sS$ and $n \in \bZ \setminus \{0\}$ since $\sS$ is an orthogonal collection in $\sD = \Db(\kk Q)$ and $F$ is an autoequivalence. For $n = 1$, we have $\Hom_{\sD} (S_1, F S_2) \simeq D \Hom_{\sD} (\Sigma^w S_2, S_1) = 0$, since $w \geq 1$ and $\sS$ is an $\infty$-orthogonal collection. 

Now consider $n=2$. Recall that $(\sX,\sY)$ denotes the standard t-structure on $\sD$. Since $S_1 \in \Fw \subseteq \Sigma^{w}  \SSS \sY$, there exists $Y \in \sY$ such that $S_1  = \Sigma^{w} \SSS Y = FY$. Hence, $\Hom_{\sD} (S_1,  F^2 S_2) \simeq D \Hom_{\sD} (\Sigma^{w} S_2, Y) = 0$, since $S_2$, and so $\Sigma^{w} S_2$, lie in $\sX$.

Now let $n > 2$. By the hereditary property, we have $F^n S_2 \in \SSS^n \Sigma^{nw} \sX \subseteq \Sigma^{nw} \sX$ and $S_1 \in \SSS \Sigma^{w} \sY \subseteq \Sigma^{w+1} \sY$. Therefore, $\Hom_{\sD} (S_1, F^n S_2) \simeq \Hom_{\sD} (Y,\Sigma^{(n-1)w-1} X)$, for some $Y \in \sY$ and $X \in \sX$. Since $(n-1) w -1 \geq 1$, it  follows that $\Hom_{\sD} (S_1, F^n S_2) = 0$, again by the hereditary property.

Finally, let $n \leq -1$. Since $S_2 \in \Sigma^{w} \SSS \sY$, we get $F^n S_2 \in \Sigma^{(n+1)w} \SSS^{n+1} \sY \subseteq \Sigma^{(n+1)w} \sY \subseteq \sY$ by the hereditary property, the fact that $\sY$ is the co-aisle of a t-structure, and $(n+1)w \leq 0$. It follows that $\Hom_{\sD} (S_1, F^n S_2) = 0$ since $S_1  \in \sX$. 

Statement $(2)$ follows from the $\infty$-orthogonality property of $\sS$ when $k = 1$, and from the argument above for $n >2$, when $k \geq 2$. 
\end{proof}

\begin{corollary} \label{cor:order}
Let $\sS$ be an $\infty$-orthogonal collection of $\Db(\kk Q)$ contained in $\cF_w$. Then $\extn{F^n \sS}_{\Db(\kk Q)} * \extn{\sS}_{\Db(\kk Q)} \subseteq \extn{\sS}_{\Db(\kk Q)} * \extn{F^n \sS}_{\Db(\kk Q)}$, for all $n \geq 1$. 
\end{corollary}

\begin{proof}
Given $D \in \extn{F^n \sS}_{\sD} * \extn{\sS}_{\sD}$, we have a triangle $\trilabels{F^n S_1}{D}{S_2}{}{}{\alpha}$, with $S_1, S_2 \in \extn{\sS}$. It follows from Lemma~\ref{lem:smc-vanishing-conditions}(2) that $\alpha = 0$. Hence, the triangle splits, and $D \simeq S_2 \oplus F^n S_1 \in \extn{\sS}_{\sD} * \extn{F^n \sS}_{\sD}$. 
\end{proof}

Following Lemma~\ref{lem:smc-vanishing-conditions} and Corollary~\ref{cor:order}, for an $\infty$-orthogonal collection $\sS$ of $\Db(\kk Q)$ contained in $\Fw$, we define the following extension-closed subcategory of $\Db(\kk Q)$:
\[
\ES \coloneqq \extn{F^n \sS \mid n \in \bZ}_{\Db(\kk Q)}.
\]
We now show that each object of $\ES$ admits a filtration by objects in $\{ F^n \sS \mid n \in \bZ\}$.

\begin{lemma}\label{lem:orderedE_S}
Let $\sS$ be an $\infty$-orthogonal collection of $\Db(\kk Q)$ contained in $\Fw$. Then $\ES$ is closed under direct summands and  
\[
\ES = \bigcup_{m<n} \extn{F^m \sS}_{\Db(\kk Q)} * \extn{F^{m+1} \sS}_{\Db(\kk Q)} * \cdots * \extn{F^n \sS}_{\Db(\kk Q)}.
\]
\end{lemma}

\begin{proof}
The first statement follows immediately from Lemma~\ref{lem:smc-vanishing-conditions}(1) and Lemma~\ref{lem:basic-properties}(1). 

By Lemma~\ref{lem:smc-vanishing-conditions}(1) and Lemma~\ref{lem:basic-properties}(2), we have $\ES =  \bigcup_{n \geq 1} (\sT)_n$, where $\sT \coloneqq \{F^n \sS \mid n \in \bZ\}$. Hence, for $D \in \ES$ there is a tower of the form
\begin{equation} \label{eq:towerE_S}
\xymatrix@!R=8px{
0 = D_0 \ar[r] & D_1 \ar[r] \ar[d] & D_2 \ar[r] \ar[d] & \cdots \ar[r] & D_{n-1} \ar[r] \ar[d] & D_n = D\ar[d] \\
&                                F^{i_1} S_{1} \ar@{~>}[ul] & F^{i_2} S_{2} \ar@{~>}[ul] & \cdots & F^{i_{n-1}} S_{n-1} \ar@{~>}[ul] & F^{i_n} S_{n} \ar@{~>}[ul]
}
\end{equation}
with $S_\ell \in \sS$ for $1 \leq \ell \leq n$. By Corollary~\ref{cor:order}, we can re-order the indices in the tower above so that $i_1 \leq i_2 \leq \ldots \leq i_n$. Therefore, $\ES \subseteq \bigcup_{m<n} \extn{F^m \sS}_{\sD} * \extn{F^{m+1} \sS}_{\sD} * \cdots * \extn{F^n \sS}_{\sD}$, where $\sD = \Db(\kk Q)$. The other inclusion is trivial. 
\end{proof}

Before we are able to prove Proposition~\ref{prop:SMC-SMS-well-defined} we need the following lemmas connecting the functorial finiteness of $\ES$ in $\Db(\kk Q)$ and the functorial finiteness of the extension closure of $\sS$ in $\Cw$.

\begin{lemma} \label{lem:extension-closure}
Let $\pi \colon \Db(\kk Q) \to \Cw$ be the natural projection functor.
If $\sS$ is an $\infty$-orthogonal collection of $\Db(\kk Q)$ contained in $\Fw$, then $\pi (\ES) = \extn{\sS}_{\Cw}$.
\end{lemma}

\begin{proof}
To see that $\pi(\ES) \subseteq \extn{\sS}_{\Cw}$, observe that, by Lemma~\ref{lem:orderedE_S}, $D \in \ES$ admits a tower as in \eqref{eq:towerE_S} with each $S_\ell \in \sS$. Since $\pi (F^{i_\ell} S_\ell) = S_\ell \in \sS$, applying $\pi$ to the tower above gives a filtration of $\pi(D)$ in $\extn{\sS}_{\Cw}$. 

Conversely, we need to show that if $D \in \pi^{-1}(\extn{\sS}_{\Cw})$ then $D \in \ES$. It is enough to check for $D$ indecomposable. Let $D \in \pi^{-1} (\extn{\sS}_{\Cw})$ be indecomposable; note that $D \in F^k \Fw$ for some $k \in \bZ$. We proceed by induction on the $\sS$-length of $\pi(D)$ in $\Cw$.
If $D \in \pi^{-1} (\sS)$, then $D = F^k S$ for some $S \in \sS$. Hence, $D \in \{F^i \sS \mid i \in \bZ\} \subseteq \ES$. 
Suppose $n > 1$. Assume, by induction, that if $D \in \pi^{-1}(\extn{\sS}_{\Cw})$  is such that $\pi(D)$ has $\sS$-length $n-1$ in $\Cw$ then $D \in \ES$ has $\{F^i \sS \mid i \in \bZ\}$-length $n-1$ in $\ES$. 
Now suppose $D \in \pi^{-1}(\extn{\sS}_{\Cw})$ is such that $\pi (D)$ has $\sS$-length $n$ in $\Cw$. Then, by Lemma~\ref{lem:basic-properties}(2), there is a triangle in $\Cw$ of the form 
\[
\tri{C}{\pi(D)}{S},
\] 
with $S \in \sS$ and $C \in \extn{\sS}_{\Cw}$ has $\sS$-length $n-1$ in $\Cw$. This triangle is the image under $\pi$ of a triangle in $\Db(\kk Q)$ of the form 
\[
\tri{F^i D'}{D}{F^j S},
\]
where $i \in \{k, k-1\}$ and $j \in \{k, k+1\}$, by Lemma~\ref{lem:morphisms}. Since $\pi(F^i D') = C$, it follows by induction that $F^i D' \in \ES$ has $\{F^i \sS \mid i \in \bZ\}$-length $n-1$ in $\Db(\kk Q)$, whence $D \in \ES$ has $\{F^i \sS \mid i \in \bZ\}$-length $n$.
\end{proof}

\begin{lemma} \label{lem:ES-funct-finite}
Let $\pi \colon \Db(\kk Q) \to \Cw$ be the natural projection functor.
Suppose $\sS$ is an $\infty$-orthogonal collection of $\Db(\kk Q)$ contained in $\Fw$. Then $\ES$  is contravariantly (resp. covariantly, functorially) finite in $\Db(\kk Q)$ if and only if 
$\extn{\sS}_{\Cw}$ is contravariantly (resp. covariantly, functorially) finite in $\Cw$. 
\end{lemma}

\begin{proof}
We only establish the contravariantly finite statements; the covariantly finite statement is similar and the functorially finite statement follows from combining both.

Suppose $\ES$ is contravariantly finite in $\sD = \Db(\kk Q)$. For $D \in \Fw$, take a minimal right $\ES$-approximation of $D$ in $\sD$ and extend it to a triangle,
\[
\trilabel{E_D}{D}{X_D}{\alpha}.
\]
By the triangulated Wakamatsu lemma (see, e.g.\ \cite[Lemma 2.1]{Jorgensen}), $X_D \in (\cE_\sS)^\perp$. Applying $\pi$ to this triangle, we get the following triangle in $\Cw$:
\[
\trilabel{\pi(E_D)}{\pi(D)}{\pi(X_D)}{\pi(\alpha)}.
\] 
By Lemma~\ref{lem:extension-closure}, we have $\pi (E_D) \in \extn{\sS}_{\Cw}$. By construction of $\Cw$, see Section~\ref{sec:hereditary}, we have $\Hom_{\Cw} (\sS, \pi (X_D)) = \bigoplus_{i \in \bZ} \Hom_\sD (F^i \sS, X_D) = 0$ since $X_D \in (\cE_\sS)^\perp$. Therefore, $\pi(\alpha)$ is a right $\extn{\sS}_{\Cw}$-approximation of $D = \pi (D)$, and so $\extn{\sS}_{\Cw}$ is  contravariantly finite in $\Cw$. 

Conversely, suppose $\extn{\sS}_{\Cw}$ is contravariantly finite in $\Cw$. Let $D \in \sD$ and consider the extension of a minimal right $\extn{\sS}_{\Cw}$-approximation of $\pi(D)$ to a triangle in $\Cw$, 
\[
\trilabel{S_D}{\pi(D)}{Y_D}{},
\]
where $\Hom_{\Cw} (\sS, Y_D) = 0$ by the triangulated Wakamatsu lemma. Since the triangulated structure on $\Cw$ is induced by that of $\sD$, this triangle is the image under $\pi$ of a triangle,
\[
\trilabel{E_D}{D}{X_D}{f}
\]
in $\sD$. It follows from Lemma~\ref{lem:extension-closure} that $E_D \in \ES$. On the other hand, we have $0 = \Hom_{\Cw}(\sS, Y_D) = \bigoplus_{i \in \bZ} \Hom_\sD(F^i \sS, X_D)$. Hence, $\Hom_\sD (F^i \sS, X_D) = 0$, for all $i \in \bZ$, and so $X_D \in (\ES)\orth$. Therefore, the map $f$ is a right $\ES$-approximation of $D$, from which it follows that $\ES$ is contravariantly finite in $\sD$.
\end{proof}

We are almost ready to prove Proposition~\ref{prop:SMC-SMS-well-defined}. First, we set up a final piece of notation. For integers $m \leq n$ and an $\infty$-orthogonal collection $\sS$, we set
\[
\ES^{[m,n]} \coloneqq \extn{F^m \sS}_{\Db(\kk Q)} * \extn{F^{m+1} \sS}_{\Db(\kk Q)} * \cdots * \extn{F^n \sS}_{\Db(\kk Q)} \subseteq \ES.
\]

\begin{proof}[Proof of Proposition~\ref{prop:SMC-SMS-well-defined}]
Let $\sS$ be a simple-minded collection in $\sD = \Db(\kk Q)$ contained in $\cF_w$. By the dual of \cite[Proposition 3.6]{IJ}, $\pi (\sS)$ is a left $w$-Riedtmann configuration in $\Cw$. Hence, by Proposition~\ref{prop:sms-char}, it only remains to check that $\extn{\sS}_{\Cw}$ is contravariantly finite in $\Cw$. Hence, by Lemma~\ref{lem:ES-funct-finite}, it suffices to show that $\ES$ is contravariantly finite in $\sD$.

Let $D \in \sD$. There are only finitely many $i \in \bZ$ with $\Hom_{\sD} (F^i \sS,D) \neq 0$ as $\kk Q$ is hereditary. Let $m, n \in \bZ$ such that $m \leq n$ and $\Hom_{\sD} (F^{<m} \sS, D) = 0 = \Hom_{\sD} (F^{>n} \sS,D )$. By Corollary~\ref{cor:extn-of-smc-funct-finite}, $\extn{\sS}_{\sD}$ is contravariantly finite in $\Db(\kk Q)$. Hence, $\ES^{[m,n]}$ is contravariantly finite in $\Db(\kk Q)$ by \cite[Lemma 5.3(1)]{SZ}. Since any component of a nonzero morphism from an object of $\ES$ to $D$ must originate from a summand $E \in \ES^{[m, n]}$, we have that any right $\ES^{[m,n]}$-approximation of $D$ is also a right  $\ES$-approximation. Hence, $\ES$ is contravariantly finite in $\sD$.
\end{proof}

\subsection{The induced map $\pi$ is surjective}

In order to complete the proof of Theorem~\ref{thm:SMC-SMS-bijection}, we have to show that the map $\pi$
induced by the natural projection functor $\pi \colon \Db(\kk Q) \to \Cw$ in Proposition~\ref{prop:SMC-SMS-well-defined} is surjective. To see this, we need a lemma, which first requires some notation. For an $\infty$-orthogonal collection $\sS$ in $\Db(\kk Q)$ and an integer $n$, we define the following subcategories of $\ES$:
\begin{align*}
\ES^{\leq n} & \coloneqq \bigcup_{i \leq n} \extn{F^i \sS}_{\Db(\kk Q)} * \extn{F^{i+1} \sS}_{\Db(\kk Q)} * \cdots * \extn{F^n \sS}_{\Db(\kk Q)}; \text{ and,} \\
\ES^{\geq n}& \coloneqq \bigcup_{i \geq n} \extn{F^n \sS}_{\Db(\kk Q)} * \extn{F^{n+1} \sS}_{\Db(\kk Q)} * \cdots * \extn{F^i \sS}_{\Db(\kk Q)}.
\end{align*}
Similarly, we also set $\ES^{< n} = \ES^{\leq n-1}$ and $\ES^{> n} = \ES^{\geq n+1}$. 

\begin{lemma}\label{lem:extnS-contra-finite-in-E_S}
Let $\sS$ be an $\infty$-orthogonal collection in $\Db(\kk Q)$. Then
\begin{enumerate}
\item $\extn{\sS}_{\Db(\kk Q)}$ is contravariantly finite in $\cE_\sS^{\leq 0}$; and, 
\item $\extn{\sS}_{\Db(\kk Q)}$ is covariantly finite in $\cE_\sS^{\geq 0}$.
\end{enumerate}
\end{lemma}
\begin{proof}
We only prove the first statement; the second statement is dual. Write $\sD = \Db(\kk Q)$. Let $D \in \ES^{\leq 0}$. 
First note that if $\Hom_{\sD} (\sS, D) = 0$, then $D$ admits a (minimal) right $\extn{\sS}_{\sD}$-approximation, namely $S_D = 0 \too D$, whose cone $X_D \simeq D \in \sS\orth$. Therefore, we may assume that $\Hom_{\sD} (\sS,D) \neq 0$. 

By Lemma~\ref{lem:orderedE_S}, $d$ admits a decomposition, 
\begin{equation} \label{decomposition}
\trilabels{E^{<0}}{D}{E^0}{\alpha}{\beta}{},
\end{equation} 
with $E^{<0} \in \ES^{<0}$ and $E^0 \in \extn{\sS}_{\sD}$. We shall construct a right $\extn{\sS}_{\sD}$-approximation of $D$ by induction on the $\sS$-length of $E^0$ in $\extn{\sS}_{\sD}$.

Suppose the $\sS$-length of $E^0$ is one, i.e.\ $E^0 \in \sS$, and suppose $\phi\colon S \too D$ is a nonzero map with $S \in \sS$. If $\beta \phi = 0$, then $\phi = 0$ because it factors through $E^{<0}$ and, by Lemma~\ref{lem:smc-vanishing-conditions}, $\Hom_{\sD} (\sS, E^{<0}) = 0$. Hence, $\beta \phi \neq 0$, and since $S, E^0 \in \sS$, $\beta \phi$ must be an isomorphism because $\sS$ is an orthogonal collection in $\sD$. Therefore, $\beta$ is a split epimorphism and $D \simeq E^0 \oplus E^{<0}$. It follows that $S \simeq E^0$ is the unique $S \in \sS$, up to isomorphism, such that $\Hom_{\sD}(S,D) \neq 0$. Hence, $E^0 \rightlabel{\begin{pmat} 1 \\ 0 \end{pmat}} D$ is a right $\extn{\sS}_{\sD}$-approximation of $D$, whose cone $X_D = E^{<0} \in \sS\orth$. 

Now suppose the $\sS$-length of $E^0$ is $n >1$, i.e.\ there is a triangle $\tri{S'}{E^0}{S}$ with $S \in \sS$, and $S' \in \extn{\sS}_{\sD}$ of $\sS$-length $n-1$. Combining this triangle with \eqref{decomposition}, we get the octahedral diagram: 
\[
\xymatrix@!R=8px{
                               & S' \ar@{=}[r] \ar[d]    & S' \ar[d]          & \\
D \ar[r] \ar@{=}[d] & E^0 \ar[r] \ar[d]            & \Sigma E^{<0} \ar[r] \ar[d] & \Sigma D \ar@{=}[d] \\
D \ar[r]                   & S \ar[r] \ar[d]     & \Sigma A \ar[r] \ar[d]  & \Sigma D \\
                               & \Sigma S' \ar@{=}[r]  & \Sigma S'.      &
}
\]
The right-hand vertical triangle gives a decomposition of $A$ in $\ES^{\leq 0}$, $\tri{E^{<0}}{A}{S'}$, in which the $\sS$-length of $S'$ is $n-1$. Hence, by induction, $A$ admits a right $\extn{\sS}_{\sD}$-approximation $S_A \too A$, whose cone $X_A$ lies in $\sS\orth$. Consider the following octahedral diagram: 
\begin{equation} \label{octahedron-2}
\xymatrix@!R=8px{
                               & \Sigma^{-1} S \ar@{=}[r] \ar[d]    & \Sigma^{-1} S \ar[d]          & \\
S_A \ar[r] \ar@{=}[d] & A \ar[r] \ar[d]            & X_A \ar[r] \ar[d] & \Sigma S_A \ar@{=}[d] \\
S_A \ar[r]^\alpha                   & D \ar[r] \ar[d]    & Y \ar[r] \ar[d]^{\beta}  & \Sigma S_A \\
                               & S \ar@{=}[r]  & S.    &
}
\end{equation}
If $Y \in \sS\orth$ then $\alpha$ is a right $\extn{\sS}_{\sD}$-approximation of $D$ with cone lying in $\sS\orth$, and we are done. So suppose $Y \not\in \sS^\perp$. 
For $S_1 \in \sS$, applying $\Hom_{\sD} (S_1, -)$ to the right-hand vertical triangle in \eqref{octahedron-2}, shows that the map $\Hom_{\sD} (S_1, \beta) \colon \Hom_{\sD} (S_1,Y) \into \Hom_{\sD} (S_1,S)$ is injective because $X_A \in \sS\orth$. Hence, if $S_1 \neq S$, then $\Hom_{\sD} (S_1, Y) = 0$. Since $Y \not\in \sS^\perp$, it follows that we must have $\Hom_{\sD} (S,Y) \neq 0$ and  $\Hom_{\sD} (S,\beta)$ is injective. Suppose $\gamma \in \Hom_{\sD} (S, Y)$ is nonzero. Then $\beta \gamma$ must be an isomorphism, which implies that $\beta$ is a split epimorphism, whence $Y \simeq S \oplus X_A$. 

Now, using the other split triangle, we get a new octahedral diagram:
\[
\xymatrix@!R=8px{
                               & S \ar@{=}[r] \ar[d]    & S \ar[d]          & \\
D \ar[r] \ar@{=}[d] & X_A \oplus S \ar[r] \ar[d]            & \Sigma S_A \ar[r] \ar[d] & \Sigma D \ar@{=}[d] \\
D \ar[r]                   & X_A \ar[r] \ar[d]    & \Sigma S_D \ar[r]_{-\Sigma g} \ar[d]  & \Sigma D\\
                               & \Sigma S \ar@{=}[r]  & \Sigma S.     &
}
\]
Clearly, $S_D \in \extn{\sS}_{\sD}$, and since $X_A \in \sS^\perp$, $g$ is a right $\extn{\sS}_{\sD}$-approximation of $D$ with cone in $\sS^\perp$. This finishes the proof.
\end{proof}

We are now ready to prove Theorem~\ref{thm:SMC-SMS-bijection}.

\begin{proof}[Proof of Theorem~\ref{thm:SMC-SMS-bijection}]
By Proposition~\ref{prop:SMC-SMS-well-defined}, the map $\pi$ is well defined and it is clearly injective since the projection functor $\pi$ gives a bijection between $\ind{\Fw}$ and $\ind{\Cw}$. It remains to show that $\pi$ induces a surjection. 

Let $\sS$ be a $w$-simple-minded system in $\Cw$. The fact that the lift of $\sS$ to $\sD = \Db(\kk Q)$, which will also be denoted by $\sS$, is an $\infty$-orthogonal collection in $\sD$ follows from the proof of \cite[Theorem 1.2]{IJ}, as this part of the proof does not require $Q$ to be Dynkin. Hence, it remains to show that $\thick{\sD}{\sS} = \sD$. 

First, we claim that $\extn{\sS}_{\sD}$ is functorially finite in $\ES$. Indeed, on the one hand we have that $\extn{\sS}_{\sD}$ is contravariantly finite in $\ES^{\leq 0}$ by Lemma~\ref{lem:extnS-contra-finite-in-E_S}. On the other hand, since $\ES = \ES^{\leq 0} * \ES^{>0}$ by Lemma~\ref{lem:orderedE_S}, and $\{F^n \sS \mid n \in \bZ\}$ is an orthogonal collection in $\sD$, we have that $\ES^{\leq 0}$ is contravariantly finite in $\ES$. Hence, by transitivity, $\extn{\sS}_{\sD}$ is contravariantly finite in $\ES$. Dually, using $\ES^{\geq 0}$,  
$\extn{\sS}_{\sD}$ is also covariantly finite in $\ES$, and therefore $\extn{\sS}_{\sD}$ is functorially finite in $\ES$. Since $\sS$ is a $w$-simple-minded system in $\Cw$, we have that $\extn{\sS}_{\Cw}$ is functorially finite in $\Cw$. It then follows by Lemma~\ref{lem:ES-funct-finite} and transitivity that $\extn{\sS}_{\sD}$ is functorially finite in $\sD$.  

We will now show that $\thick{\sD}{\sS} = \sD$. Since $\extn{\sS}_{\sD}$ is functorially finite in $\sD$, we have that $\extn{\sS}_{\sD}$ is contravariantly finite in $(\Sigma^{> 0} \sS)\orth$ and covariantly finite in ${}\orth (\Sigma^{<0} \sS)$. And since $\kk Q$ is hereditary, for any $D \in \sD = \Db(\kk Q)$, we have $\Hom_{\sD}(\Sigma^i D, \sS) = 0 = \Hom_{\sD}(\Sigma^i \sS, D)$ for $i \gg 0$. Therefore, by \cite[Proposition 3.2]{Jin2}, we have that $({}\orth (\Sigma^{<0} \sS),\cosusp \sS)$ is a t-structure in $\sD$.

Without loss of generality, we may assume that $0 \neq D \in \extn{\Sigma^{\leq -w-2} \mod{\kk Q}}_{\sD}$. As a consequence, $\Hom_{\sD}(D, \Sigma^{\geq 0} \sS) = 0$. Consider the decomposition triangle of $D$ with respect to the t-structure $({}\orth (\Sigma^{<0} \sS),\cosusp \sS)$:
\[
\tri{X_D}{D}{Y_D}.
\]
The morphism $X_D \to D$ is thus a minimal right ${}\orth (\Sigma^{<0} \sS)$-approximation, which is therefore zero if and only if $X_D \simeq 0$. Suppose $X_D \neq 0$, then the morphism $X_D \to D$ must be nonzero. Since $\kk Q$ is hereditary, this means that $X_D \in \extn{\Sigma^{\leq -w-2} \mod{\kk Q}}_{\sD}$ also, in which case, $\Hom_{\sD}(X_D, \Sigma^{\geq 0} \sS) = 0$. Hence $X_D \in {}\orth (\Sigma^{\bZ} \sS)$. Now, by the final part of the proof of \cite[Theorem 1.2]{IJ}, we can conclude that $X_D \simeq 0$. Hence $D \simeq Y_D \in \cosusp \sS \subseteq \thick{\sS}{\sD}$. It follows that $\thick{\sD}{\sS} = \sD$, as required.
\end{proof}

\section{Sincere orthogonal collections} \label{sec:sincere}

Let $Q$ be an acyclic quiver. In this section we establish a bijection between $w$-simple-minded systems in $\Cw(\kk Q)$ and sincere $\infty$-orthogonal collections of $\Db(\kk Q)$ sitting in some truncation of the fundamental domain of $\Cw(\kk Q)$; see Theorem~\ref{thm:w-sincere}. This result generalises \cite[Theorem 4.8]{CS12}, which established the same result in the case that $w=1$ and $Q$ is simply-laced Dynkin.
In Section~\ref{sec:noncrossing}, we will use Theorem~\ref{thm:w-sincere} to establish a bijection between $w$-simple-minded systems and positive $w$-noncrossing partitions in the Weyl group of the corresponding type. 
However, we believe that Theorem~\ref{thm:w-sincere} holds independent representation-theoretic interest.

Before proceeding, we require some background on exceptional sequences.

\subsection{Exceptional sequences}

The notion of exceptional sequence goes back to the Moscow school in the 1980s; see e.g.\ \cite{Rudakov}. Recall from \cite{Bondal} that an object $e$ in a triangulated category $\sD$ is called \emph{exceptional} if $\Hom_\sD(e, \Sigma^i e) = 0$ for all $i \neq 0$ and $\Hom_\sD(e,e) \simeq \kk$. An ordered collection of exceptional objects $\sE = (e_1,\ldots,e_r)$ of $\sD$ is called an \emph{exceptional sequence} if $\Hom_\sD(e_j, \Sigma^i e_k) = 0$ for all $i \in \bZ$ and $j > k$.
The exceptional sequence $\sE$ is called \emph{complete} if $\thick{\sD}{\sE} = \sD$.
If $Q$ is an acyclic quiver with $n$ vertices and $\sE = (E_1,\ldots,E_r)$ is an exceptional sequence in $\Db(\kk Q)$ then $\sE$ is complete if and only if $r = n$; see \cite{BRT,CB,Ringel}.

Exceptional sequences are often called `exceptional collections'. In this article we avoid this term to avoid ambiguity when considering orthogonal collections.

The following lemma gives a link between exceptional sequences and simple-minded collections in $\Db(\kk Q)$.

\begin{lemma} \label{lem:exceptional}
Let $Q$ be an acyclic quiver with $n$ vertices. If $\sS$ is a simple-minded collection in $\Db(\kk Q)$ then the objects of $\sS$ can be ordered into a (complete) exceptional sequence $\sE = (E_1,\ldots,E_n)$ in which the cohomological degrees of the $E_i$ 
are weakly decreasing.
\end{lemma}

\begin{proof}
By \cite[Lemma 2.3]{BRT}, the objects of a $\Hom_{\leq 0}$-configuration (for the definition, see \cite[\S 2.2]{BRT}) can be ordered into a complete exceptional sequence $\sE = (E_1,\ldots,E_n)$ in which the cohomological degrees of the $E_i$ are weakly decreasing, and therefore strongly generate $\sD = \Db(\kk Q)$. Hence the set of $\Hom_{\leq 0}$-configurations in $\sD$ is a subset of the set of simple-minded collections.
By \cite[Theorem 2.4]{BRT}, the set of $\Hom_{\leq 0}$-configurations in $\sD$ is in bijection with the set of silting objects in $\sD$. However, the map from the set of $\Hom_{\leq 0}$-configurations in $\sD$ to the set of silting objects in $\sD$ constructed in \cite{BRT} coincides with the map from simple-minded collections in $\sD$ to silting objects in $\sD$; see \cite[\S 5.6]{KY}. Hence the set of $\Hom_{\leq 0}$-configurations in $\sD$ coincides with the set of simple-minded collections in $\sD$, giving the lemma.
\end{proof}

\subsection{Sincere orthogonal collections}

Let $Q$ be an acyclic quiver and $w \geq 1$ be an integer. Write $\sH \coloneqq \mod{\kk Q}$ and consider the functor
\begin{equation} \label{cohomology}
H \colon \Db(\kk Q) \too \sH \text{ given by } X \longmapsto \bigoplus_{i \in \bZ} H^i(X),
\end{equation}
where $H^i(-)$ denotes the $i^{\rm th}$-cohomology of $X$ with respect to the standard t-structure $(\sX,\sY)$ in $\Db(\kk Q)$.

Recall that a module $M \in \sH$ is called \emph{sincere} if $\Hom_\sH(P,M) \neq 0$ for any projective module $P$. A set of modules $\sX \subseteq \sH$ will be called a \emph{sincere set of modules} if for each projective module $P$ there exists a module $X \in \sX$ such that $\Hom_\sH(P,X) \neq 0$. If $\sX$ is a finite set, then $\sX$ is a sincere set of modules if and only if $\bigoplus_{X \in \sX} X$ is a sincere module. 
Note that there are equivalent formulations of sincerity using injective modules.
Using sincerity, we make the following definition.

\begin{definition}
An $\infty$-orthogonal collection $\sS \subseteq \Db(\kk Q)$ will be called \emph{$w$-sincere} if $\sS \subseteq \sX \cap \Sigma^w \sY \subseteq \Fw$ and $\{H(S) \mid S \in \sS\}$ is a sincere set of modules in $\sH$.
It will be called \emph{exceptionally finite} if the objects of $\sS$ can be ordered into an exceptional sequence and $\extn{\sS}_{\Db(\kk Q)}$ is functorially finite in $\Db(\kk Q)$.
\end{definition}

Recall $\pi \colon \Db(\kk Q) \to \Cw$ is the canonical projection functor.

\begin{proposition} \label{prop:w-sincere}
Let $Q$ be an acyclic quiver and $w \geq 1$ an integer. Then there is a well-defined map
\begin{align*}
\Theta \colon 
\left\{ 
\parbox{5.25cm}{\centering
$w$-Riedtmann configurations in $\Cw$}
\right\} 
& \too 
\left\{ 
\parbox{4.75cm}{\centering
$w$-sincere $\infty$-orthogonal collections in $\Db(\kk Q)$}
\right\}, \\
\sS 
& \longmapsto
\pi^{-1}(\sS) \cap \sX \cap \Sigma^w \sY
\end{align*}
which restricts to a well-defined map
\[
\Theta \colon
\left\{ 
\parbox{3.25cm}{\centering
$w$-simple-minded systems in $\Cw$}
\right\} 
\too 
\left\{ 
\parbox{8cm}{\centering
exceptionally finite, $w$-sincere $\infty$-orthogonal collections in $\Db(\kk Q)$}
\right\}.
\]
\end{proposition}

\begin{proof}
Let $\sS$ be a $w$-Riedtmann configuration in $\Cw$. Partition the lift of $\sS$ into $\Fw$, $\pi^{-1}(\sS) \cap \Fw$, into $\sR \cup \sT$, where $\sR \subseteq \sX \cap \Sigma^w \sY$ and $\sT \subseteq \Sigma^w \sX \cap \SSS \Sigma^w \sY$. We need to show that $\sR$ is a $w$-sincere $\infty$-orthogonal configuration. As in the proof of Theorem~\ref{thm:SMC-SMS-bijection}, we can invoke \cite[Theorem 1.2]{IJ} to see that $\sR$ is $\infty$-orthogonal. By definition $\sR \subseteq \sX \cap \Sigma^w \sY$, so it only remains to show the sincerity part of the definition. 

Suppose, for a contradiction, that $\{H(R) \mid R \in \sR\}$ is not a sincere set. This means that there exists an indecomposable projective module $P$ such that $\Hom_{\sD}(P, \Sigma^{-i} R) = 0$ for each $R \in \sR$ and $0 \leq i \leq w-1$, where $\sD = \Db (\kk Q)$. 
Now, since $P \in \Fw$ and $\sR \subseteq \Fw$, by Lemma~\ref{lem:morphisms} for $0 \leq i \leq w-1$, we have
\begin{align*}
\Hom_{\Cw}(P,\Sigma^{-i}R) & \simeq \Hom_{\sD}(P,\Sigma^{-i} R) \oplus D \Hom_{\sD}(R, \Sigma^{i-w} P) \\
                                              & \simeq D \Hom_{\sD}(R, \Sigma^{i-w} P).
\end{align*}
Since $i - w < 0$ for each $i$, $\sR \subseteq \Fw$ and $P \in \sH$, the hereditary property means that $\Hom_{\sD}(R, \Sigma^{i-w} P) = 0$. Hence, $\Hom_{\Cw}(P,\Sigma^{-i}R) = 0$ for each $R \in \sR$ and each $0 \leq i \leq w-1$.

Now each $T \in \sT$ can be written as $T = \Sigma^w T'$ for some $T' \in \sH$. Again, using Lemma~\ref{lem:morphisms}, for $0 \leq i \leq w-1$, we have
\begin{align*}
\Hom_{\Cw}(P,\Sigma^{-i} T) & \simeq \Hom_{\sD}(P, \Sigma^{-i} T) \oplus D \Hom_{\sD}(T, \Sigma^{i-w} P) \\
                                             & \simeq \Hom_{\sD}(P, \Sigma^{w-i} T') \oplus D \Hom_{\sD}(T', \Sigma^{i-2w} P).
\end{align*}
Since $w - i \geq 1$ and $P$ is projective, we have $\Hom_{\sD}(P, \Sigma^{w-i} T') =0$. Since $i - 2w < 0$, and $\sH$ is hereditary, we have  $\Hom_{\sD}(T', \Sigma^{i-2w} P) = 0$. It follows that $\Hom_{\Cw}(P,\Sigma^{-i} T) = 0$ for each $0 \leq i \leq w-1$ and $T \in \sT$. Since $0 \neq P \in \Cw$, this contradicts the fact that $\sS$ is a $w$-Riedtmann configuration. Hence, $\{H(R) \mid R \in \sR\}$ must be a sincere set, as required.

Finally, to see the restriction, suppose further that $\sS$ is a $w$-simple-minded system in $\Cw$. Then $\extn{\sR}_{\Cw}$ is functorially finite in $\extn{\sS}_{\Cw}$ by Lemma~\ref{lem:basic-properties}$(3)$. Hence, by Proposition~\ref{prop:sms-char} and transitivity of functorial finiteness, $\extn{\sR}_{\Cw}$ is functorially finite in $\Cw$. By Lemmas~\ref{lem:ES-funct-finite} and \ref{lem:extnS-contra-finite-in-E_S}, we see that $\extn{\sR}_{\sD}$ is functorially finite in $\sD$.
By Theorem~\ref{thm:SMC-SMS-bijection} and Lemma~\ref{lem:exceptional}, the objects of $\sS$ and hence $\sR$ can be ordered into an exceptional sequence, making $\sR$ exceptionally finite, as required.
\end{proof}

Putting together Proposition~\ref{prop:w-sincere} with Theorem~\ref{thm:SMC-SMS-bijection} we obtain the following corollary.

\begin{corollary} \label{cor:w-sincere}
Let $Q$ be an acyclic quiver and $w \geq 1$ an integer. Then there is a well-defined map
\begin{align*}
\Theta \colon 
\left\{ 
\parbox{5.25cm}{\centering
simple-minded collections of $\Db(\kk Q)$ contained in $\Fw$}
\right\} 
& \too 
\left\{ 
\parbox{6.5cm}{\centering
exceptionally finite, $w$-sincere $\infty$-orthogonal collections in $\Db(\kk Q)$}
\right\}. \\
\sS 
& \longmapsto
\sS \cap \sX \cap \Sigma^w \sY
\end{align*}
\end{corollary}

We now aim to show that the map $\Theta$ defined in Corollary~\ref{cor:w-sincere} is a bijection. In order to establish this we need a special case of Jin's reduction of simple-minded collections \cite[Theorem 3.1]{Jin2}.

\subsection{Reduction of simple-minded collections}

In this section we present an explicit specialisation of \cite[Theorem 3.1]{Jin2}. The following statement is related to \cite[Theorem 3.1]{Jin2} in a way analogous to that Aihara-Iyama's silting reduction theorem \cite[Theorem 2.37]{AI} is related to Iyama-Yang's silting reduction theorem \cite[Theorem 3.7]{IYa}. 
We include a proof for the convenience of the reader.  

\begin{proposition}[Specialisation of {\cite[Theorem 3.1]{Jin2}}] \label{prop:reduction}
Let $\sD$ be a Hom-finite, Krull-Schmidt, $\kk$-linear triangulated category. Suppose $\sT$ is an $\infty$-orthogonal collection such that $\thick{\sD}{\sT}$ is functorially finite in $\sD$. Then there is a bijection
\[
\Phi \colon \left\{
\parbox{5cm}{\centering
simple-minded collections in $\sD$ containing $\sT$}
\right\}
\rightlabel{1-1}
\left\{
\parbox{5cm}{\centering
simple-minded collections in $(\Sigma^{\bZ} \sT)\orth$}
\right\}.
\]
\end{proposition}

In Appendix~\ref{sec:reduction} we include another proof which does not rely on the functorial finiteness of the subcategory generated by $\sT$ or use Verdier localisation.

\begin{proof}
The following argument is analogous to \cite[Theorem 2.37]{AI}.

Let $\sS$ be a simple-minded collection in $\sD$ containing $\sT$. For each $s \in \sS$, consider the truncation triangle coming from the stable t-structure $(\thick{\sD}{\sT}, (\Sigma^{\bZ} \sT)\orth)$,
\begin{equation} \label{phi}
\tri{t_s}{s}{Ls},
\end{equation}
where $L \colon \sD \to (\Sigma^{\bZ} \sT)\orth$ is the left adjoint to the inclusion $(\Sigma^{\bZ} \sT)\orth \into \sD$.
Define $\Phi(\sS) \coloneqq \sR = \{ Ls \mid s \in \sS\}$. Note that $Ls \neq 0$ if and only if $s \in \sS \setminus \sT$. By truncating $t_s$ with respect to the bounded t-structure $(\susp \sT,\cosusp \Sigma^{-1}\sT)$ in $\thick{\sD}{\sT}$, observe that the left $(\cosusp \Sigma^{-1}\sT)$-approximation of $t_s$ in $\thick{\sD}{\sT}$ also gives rise to a right $(\thick{\sD}{\sT})$-approximation of $s$, which shows that $t_s \in \cosusp \Sigma^{-1} \sT$. Take $s_1,s_2 \in \sS$ and consider the corresponding truncation triangle $\tri{t_i}{s_i}{Ls_i}$ for $i = 1,2$. Applying the functor $\Hom_\sD(-,Ls_2)$ to the truncation triangle for $s_1$ and the functor $\Hom_\sD(s_1,-)$ to the truncation triangle for $s_2$ shows that $\sR$ is an $\infty$-orthogonal collection in $(\Sigma^{\bZ} \sT)\orth$; here it is important to use the fact that $t_i \in \cosusp \Sigma^{-1} \sT$.

Suppose $d \in (\Sigma^{\bZ} \sT)\orth$. Since, $\sS$ is a simple-minded collection in $\sD$ there are integers $n \geq m$ such that $d \in \Sigma^n \extn{\sS}_\sD * \cdots * \Sigma^m \extn{\sS}_\sD$. Applying the functor $L$ to the corresponding tower for $d$ shows that $d \in \Sigma^n \extn{\sR}_{(\Sigma^{\bZ}\sT)\orth} * \cdots * \Sigma^m \extn{\sR}_{(\Sigma^{\bZ}\sT)\orth}$. Hence $(\Sigma^{\bZ}\sT)\orth = \thick{(\Sigma^{\bZ}\sT)\orth}{\sR}$ and $\sR$ is a simple-minded collection in $(\Sigma^{\bZ}\sT)\orth$. This shows that the map $\Phi$ is well defined.

We now construct a map
\[
\Psi \colon \left\{
\parbox{5cm}{\centering
simple-minded collections in $(\Sigma^{\bZ} \sT)\orth$}
\right\}
\too
\left\{
\parbox{5cm}{\centering
simple-minded collections in $\sD$ containing $\sT$}
\right\}.
\]
First note that $\cosusp \sT$ is covariantly finite in $\sD$ since $\cosusp \sT$ is covariantly finite in $\thick{\sD}{\sT}$, which is in turn functorially finite in $\sD$. Hence, there is an (unbounded) t-structure $({}\orth (\Sigma^{\leq 0} \sT),\cosusp \sT)$ in $\sD$. Let $\sR$ be a simple-minded collection in $(\Sigma^{\bZ} \sT)\orth$. For each $r \in \sR$ take the truncation triangle with respect to the t-structure $({}\orth (\Sigma^{\leq 0} \sT),\cosusp \sT)$,
\begin{equation} \label{psi}
\Sigma^{-1} t_r \to s_r \to r \to t_r,
\end{equation}
and set $\Psi(\sR) = \sT \cup \{s_r \mid r \in \sR\}$. In a manner analogous to the argument above, one can show that $\Psi(\sR)$ is an $\infty$-orthogonal collection in $\sD$. For generation, take $d \in \sD$ and truncate with respect to the stable t-structure $(\thick{\sD}{\sT}, (\Sigma^{\bZ} \sT)\orth)$,
\[
\tri{t_d}{d}{Ld},
\]
and observe that $Ld \in \thick{(\Sigma^{\bZ}\sT)\orth}{\sR} = \thick{\sD}{\sR}$. It follows immediately that $d \in \thick{\sD}{\Psi(\sR)}$, whence $\Phi(\sR)$ is a simple-minded collection in $\sD$, showing that $\Psi$ is well defined.

Finally, to see that $\Phi$ and $\Psi$ are mutually inverse, applying $\Hom_\sD(-,t)$ for $t \in \cosusp \sT$ to \eqref{phi} reveals that the morphism $Ls \to \Sigma t$ is a left $(\cosusp \sT)$-approximation of $Ls$, and applying $\Hom_\sD(t,-)$ for $t \in \thick{\sD}{\sT}$ to \eqref{psi} reveals that $\Sigma^{-1} t_r \to s_r$ is a right $(\thick{\sD}{\sT})$-approximation of $s_r$.
\end{proof}

\subsection{Bijectivity of $\Theta$}

To establish the bijectivity of the map occurring in Corollary~\ref{cor:w-sincere} we require a notion of perpendicular category for abelian categories which is compatible with the one for derived categories. We recall the following from \cite{GL,Schofield}. 

\begin{definition} \label{def:perp}
Let $\sH$ be an abelian category and $\sE$ a collection of objects in $\sH$. We define the \emph{right perpendicular category} of $\sE$ by
\[
\sE\orthH \coloneqq \{ X \in \sH \mid \Hom_\sH(\sE,X) = 0 = \Ext^1_\sH(\sE,X) \}.
\]
There is an analogous definition of \emph{left perpendicular category}, ${}\orthH \sE$.
\end{definition}

\begin{lemma} \label{lem:perp}
Let $Q$ be an acyclic quiver with $n$ vertices, $\sH = \mod{\kk Q}$ and $\sE = \{E_1,\ldots,E_k\}$ be a collection of exceptional $\kk Q$-modules for $k \leq n$. Then $(\Sigma^{\bZ} \sE)\orth$ is equivalent to $\Db(\kk Q')$, where $Q'$ is an acyclic quiver with $n - k$ vertices such that $\sE\orthH \simeq \mod{\kk Q'}$.
Moreover, $\sH \cap (\Sigma^{\bZ} \sE)\orth = \sE\orthH$.
\end{lemma}

\begin{proof}
Let $\sH = \mod{\kk Q}$ and $\sD = \Db (\kk Q)$. For $X \in \sH$, we have $\Hom_{\sD}(\Sigma^i \sE,X) = 0$ for all $i \in \bZ \setminus \{-1,0\}$ since $\sE \subseteq \sH$ and $\sH$ is hereditary.
Therefore, for $X \in \sH$ we have $X \in (\Sigma^{\bZ} \sE)\orth$ if and only if $X \in \sE\orthH$.
In particular, it follows that $\Sigma^i \sH \cap  (\Sigma^{\bZ} \sE)\orth = \Sigma^i  \sE\orthH$ for each $i \in \bZ$.
Now, by \cite[Theorem 2.5]{Schofield}, there is an equivalence of (abelian) categories $\sE\orthH \to \mod{\kk Q'}$. Since $\sH$ is hereditary, each object in $\sD$ decomposes into a direct sum of its cohomology, and thus this equivalence induces an equivalence of triangulated categories $(\Sigma^{\bZ} \sE)\orth \to \Db(\kk Q')$.
\end{proof}

\begin{theorem} \label{thm:w-sincere}
Let $Q$ be an acyclic quiver with $n$ vertices and $w \geq 1$ an integer. Then there is a bijection
\begin{align*}
\Theta \colon 
\left\{ 
\parbox{5.25cm}{\centering
simple-minded collections of $\Db(\kk Q)$ contained in $\Fw$}
\right\} 
& \rightlabel{1-1} 
\left\{ 
\parbox{6.7cm}{\centering
exceptionally finite, $w$-sincere $\infty$-orthogonal collections in $\Db(\kk Q)$}
\right\}. \\
\sS 
& \longmapsto
\sS \cap \sX \cap \Sigma^w \sY
\end{align*}
\end{theorem}

\begin{proof}
The strategy of the proof is based on \cite[\S 4]{CS12}. The map $\Theta$ is well defined by Corollary~\ref{cor:w-sincere}. We first show that $\Theta$ is surjective.

Let $\sT=\{T_1,\ldots,T_k\}$ be an exceptionally finite, $w$-sincere $\infty$-orthogonal collection in $\sD = \Db(\kk Q)$. Since the objects of $\sT$ can be ordered into an exceptional sequence we have $k \leq n$.
To see that $\Theta$ is surjective, we need to find a simple-minded collection $\sS$ such that $\sT \subseteq \sS$ and $\sS \setminus \sT \subseteq \Sigma^{w+1} (\sH \setminus \inj{\sH})$. 
Let $\sE = H(\sT)$, where $H$ is the standard cohomology functor; see \eqref{cohomology}. By \cite[Theorem 3.2]{Bondal}, the subcategory $\thick{\sD}{\sT} = \thick{\sD}{\sE}$ is functorially finite in $\sD$.
Hence, by Proposition~\ref{prop:reduction}, we are looking for simple-minded collections $\sR$ in $(\Sigma^{\bZ} \sT)\orth = (\Sigma^{\bZ} \sE)\orth$ such that $\Psi(\sR) \setminus \sT \subseteq \Sigma^{w+1} (\sH \setminus \inj{\sH})$, where $\Psi$ is defined in the proof of Proposition~\ref{prop:reduction}.

Now, by Lemma~\ref{lem:perp}, $(\Sigma^{\bZ} \sT)\orth \simeq \Db(\kk Q')$ for some acyclic quiver $Q'$ with $n - k$ vertices. Let $\sR'$ be image of the set of simple $\kk Q'$-modules in $(\Sigma^{\bZ} \sT)\orth$ under the equivalence, and set $\sR = \Sigma^{w+1} \sR'$. Again, by Lemma~\ref{lem:perp}, $\sR \subseteq \Sigma^{w+1} \sH \cap (\Sigma^{\bZ} \sT)\orth$. Since $\sT \subseteq \sX \cap \Sigma^w \sY$, by the hereditary property we have that the object $T_R$ in \eqref{psi} is zero for each $R \in \sR$, whence $S_R = E$ and $\Psi(\sR) \setminus \sT = \sR$. Finally, since $\sT$ is $w$-sincere, for each injective $\kk Q$ module $I$ we have $\Hom_{\sD}(\sT,\Sigma^{w+1}I) \neq 0$, whence for each $R \in \sR$ we have $R \notin \Sigma^{w+1}(\inj{\kk Q})$. It follows that $\Theta$ is surjective.

To see that $\Theta$ is injective, suppose $\sS$ is a simple-minded collection of $\sD$ contained in $\Fw$ such that $\Theta(\sS) = \sT$. Let $\sP = \sS \setminus \sT$ and note that $\sP \subseteq \Sigma^{w+1} \sH \cap (\Sigma^{\bZ} \sT)\orth$. It follows, by the hereditary property, that for each $P \in \sP$ the object $T_P$ in the truncation triangle \eqref{phi} must be zero, whence $\Phi(\sP) = \sP$. In particular, by Proposition~\ref{prop:reduction} and Lemma~\ref{lem:perp}, $\sP$ is a simple-minded collection of $(\Sigma^{\bZ} \sT)\orth$ such that $\sP \subseteq \Sigma^{w+1} \sE\orthH$, i.e. is concentrated in one degree with respect to the standard heart $\sE\orthH \simeq \mod{\kk Q'}$ in $(\Sigma^{\bZ} \sT)\orth \simeq \Db(\kk Q')$. Therefore, the objects of $\sP$ can be ordered into an exceptional sequence by Lemma~\ref{lem:exceptional}. It now follows by \cite[Theorem 3]{Ringel}, which asserts that the unique orthogonal complete exceptional sequence in $\mod{\kk Q'}$ is that consisting of the simple $\kk Q'$-modules. Hence, $\sP = \sR$ and $\Theta$ is injective.
\end{proof}

\section{Noncrossing partitions} \label{sec:noncrossing}

Let $Q$ be an acyclic quiver. In this section, we establish a bijection between $w$-simple-minded systems in $\Cw(\kk Q)$ and positive $w$-noncrossing partitions of the corresponding Weyl group $W_Q$. This bijection generalises the existing bijection in the case that $Q$ is Dynkin, see \cite{BRT,CS12,IJ}.
The strategy follows that of \cite{CS12}.
We start with a brief review of the Weyl group of an acyclic quiver and noncrossing partitions. The main reference for the Weyl group is \cite{Ringel}, while for noncrossing partitions we refer to \cite{Armstrong}. 

\subsection{Weyl group and noncrossing partitions}

Consider the symmetric bilinear form on the Grothendieck group $K_0 (\kk Q)$ defined by
\[
([X],[Y]) \coloneqq \dim \Hom_{\kk Q} (X,Y) + \dim \Hom_{\kk Q} (Y,X) - \dim \Ext^1_{\kk Q} (X,Y) - \dim \Ext^1_{\kk Q} (Y,X),
\]
for $X, Y \in \mod{\kk Q}$. Given $X \in K_0 (\kk Q)$ with $(X,X) \in \{-2,-1,1,2\}$, the {\it reflection $t_X$ along $X$} is the isometry
\[ t_X
\colon K_0 (\kk Q) \too K_0 (\kk Q), \quad
t_X (Y) = Y - \frac{2 (Y,X)}{(X,X)} X .
\]

Let $S_1, \ldots, S_n$ be the simple $\kk Q$-modules and note that $([S_i],[S_i])=2$. The \emph{Weyl group $W_Q$ associated to the quiver $Q$} is the group of isometries generated by $R \coloneqq \{t_{[S_1]}, \ldots, t_{[S_n]}\}$. The set $R$ is called the set of \emph{simple reflections} of $W_Q$, and the set of all reflections in $W_Q$ is denoted by $T$. 

The \emph{absolute length} of $w \in W_Q$, denoted by $\ell_T (w)$, is the minimum length of $w$ written as a product of reflections. We call a minimum length expression for $w$ written as a product of reflections a \emph{$T$--reduced expression of $w$}; we denote the set of all such expressions by $\Red_T (w)$. 

A \emph{parabolic subgroup} of $W_Q$ is a subgroup of $W_Q$ generated by a subset of $R$. 
The following result, which holds for any Coxeter group $W$ of finite rank, will be useful later. 

\begin{theorem}[{\cite[Theorem 1.4]{BDSW}}]\label{thm:parabolic}
Let $W'$ be a parabolic subgroup of $W$. For each $u \in W'$, we have $\Red_T (u) = \Red_{T'} (u)$, where $T' = W' \cap T$ is the set of reflections in $W'$. 
\end{theorem}

Recall, e.g.\ from \cite[\S 3.1]{Ingalls-Thomas}, that a \emph{Coxeter element of $W_Q$} is the product of all the simple reflections in some order; in \cite{Armstrong}, a Coxeter element is called a \emph{standard Coxeter element}.
From now on, we will fix a Coxeter element $c \in W_Q$ such that the ordering of the product of simple reflections giving rise to $c$ corresponds to an ordering of the simple $\kk Q$-modules into an exceptional sequence.
Note that $\ell_T (c) = n$. We can now define (positive) $w$-noncrossing partitions for $w \geq 1$;  see \cite{Armstrong}.

\begin{definition}
Let $w \geq 1$, $\su = (u_1, \ldots, u_{w+1})$ be a $(w+1)$-tuple of elements of $W_Q$ and $c$ be the Coxeter element fixed above. The tuple $\su$ is said to be
\begin{enumerate}[label=(\roman*)]
\item a \emph{$w$-noncrossing partition} if $c = u_1 u_2 \cdots u_{w+1}$ and $n = \ell_T (c) = \ell_T (u_1) + \cdots + \ell_T (u_{w+1})$; and, 
\item a \emph{positive $w$-noncrossing partition} if it is a $w$-noncrossing partition such that the product $u_2 \cdots u_{w+1}$ does not lie in any proper parabolic subgroup. 
\end{enumerate}
The set of (positive) $w$-noncrossing partitions of $W_Q$ with respect to $c$ will be denoted by 
$\NC_w^{(+)} (W_Q)$, with the fixed Coxeter element $c$ implicitly understood.
\end{definition}
 
We end this section with the following useful connection between exceptional sequences and expressions for the Coxeter element.
 
\begin{theorem}[{\cite[Theorem 4.1]{IS}}] \label{thm:coxeter-exceptional}
Let $Q$ be an acyclic quiver. Then the following assertions hold.
\begin{enumerate}
\item For any exceptional $\kk Q$-module $E$, we have $t_{[E]} \in T$. 
\item If $(X_1, \ldots, X_n)$ is a complete exceptional sequence in $\mod{\kk Q}$, setting $t_i = t_{[X_i]}$, we have $c = t_1 \cdots t_n$. 
\item If $c = t_1 \cdots t_n$, with $t_i \in T$, then $t_i = t_{[X_i]}$, for some exceptional $\kk Q$-module $X_i$,  and $(X_1, \ldots, X_n)$ is a complete exceptional sequence. 
\end{enumerate}
\end{theorem}

It follows from Theorem~\ref{thm:coxeter-exceptional} that the set of (positive) noncrossing partitions does not depend on the ordering of the set of simple $\kk Q$-modules into an exceptional sequence.

\subsection{Simple-minded collections and noncrossing partitions}

For an acyclic quiver $Q$ and an integer $w \geq 1$, we recall the construction of a bijective map
\begin{equation} \label{the-map}
\phi \colon 
\parbox{2cm}{\centering
$\NC_w (W_Q)$}
\rightlabel{1-1}
\left\{ 
\parbox{6.75cm}{\centering
simple-minded collections of $\Db(\kk Q)$ contained in $\sX \cap \Sigma^{w+1} \sY$}
\right\}
\end{equation}
from \cite[Theorem 7.3]{BRT}; see \cite{CS12} for a similarly constructed map in the case that $Q$ is Dynkin.
Let $\su = (u_1, \ldots, u_{w+1}) \in \NC_w (W_Q)$. We construct $\phi(\su)$ in two steps.
\begin{itemize}
\item We have $c = u_1 \cdots u_{w+1}$ and $n = \ell_T (u_1) + \ell_T (u_2) + \cdots + \ell_T (u_{w+1})$. For each $1 \leq i \leq w+1$, pick a $T$-reduced expression for $u_i$. The ordered product of these expressions gives rise to a $T$-reduced expression for $c$ and by Theorem~\ref{thm:coxeter-exceptional}(3) we get a complete exceptional sequence $\sE = (\sE_1,\ldots,\sE_{w+1})$ in $\mod{\kk Q}$, where $\sE_i$ is the subsequence of $\sE$ corresponding to the $T$-reduced expression of $u_i$.
\item Let $\cC_i$ be the smallest wide (=exact abelian extension-closed) subcategory of $\mod{\kk Q}$ containing $\sE_i$. By \cite[Lemma 5]{CB}, $\cC_i$ is equivalent to $\mod{\kk Q_i}$ for some acyclic quiver $Q_i$ with $\ell_T(u_i)$ vertices. Let $\sS_i$ be the set of simples in $\cC_i$.
Then we define
\[
\phi (\su) \coloneqq \bigcup_{i = 1}^{w+1} \Sigma^{w+1-i} \sS_i.
\]
The fact that $\phi (\su)$ is independent of the choice of $T$-reduced expression follows from~{\cite[Theorem 4.3]{IS}}.
\end{itemize}

We now come to the main theorem of this section.
The case when $Q$ is Dynkin and $w=1$ was given in \cite[Theorem 5.7]{CS12}. The case when $Q$ is Dynkin and $w \geq 1$ is an integer was established in \cite[Theorem 7.4]{BRT} (see also \cite[Theorem 1.1]{IJ}), via bijections with $m$-clusters, and using a different description of positive $w$-noncrossing partitions.
Before proceeding, we require a definition.

For a $\kk Q$-module $X$, the \emph{support of $X$} is $\supp{X} \coloneqq \{a \in Q_0 \mid \Hom (P_a, X) \neq 0 \}$. 
Equivalently, $\supp{X} = \{a \in Q_0 \mid S_a \text{ occurs in a composition series for } X\}$.
For a set $\sX$ of $\kk Q$-modules the \emph{support of $\sX$} is $\supp{\sX} = \bigcup_{X \in \sX} \supp{X}$.  

\begin{theorem} \label{thm:noncrossing}
Let $Q$ be an acyclic quiver. The map $\phi$ defined in \eqref{the-map} restricts to a bijection
\[
\phi \colon 
\parbox{2cm}{\centering
$\NC_w^+ (W_Q)$}
\rightlabel{1-1} 
\left\{ 
\parbox{6.75cm}{\centering
simple-minded collections of $\Db(\kk Q)$ contained in $\Fw$}
\right\}. 
\]
\end{theorem}

\begin{proof}
Let $\su = (u_1, u_2, \ldots, u_{w+1}) \in \NC_w^+ (W_Q)$. Then, by the construction of the map $\phi$ in \eqref{the-map}, $\phi (\su) = \bigcup_{i = 1}^{w+1} \Sigma^{w+1-i} \sS_i$ is a simple-minded collection of $\Db(\kk Q)$ lying in $\sX \cap \Sigma^{w+1} \sY$.
To see that the restriction is well defined, by Theorem~\ref{thm:w-sincere}, it suffices to check that $\bigcup_{i = 2}^{w+1} \Sigma^{w+1-i} \sS_i \subseteq \sX \cap \Sigma^w \sY$ is an exceptionally finite, $w$-sincere $\infty$-orthogonal collection in $\Db(\kk Q)$. 
This set is clearly $\infty$-orthogonal.
Functorial finiteness and exceptionality follow immediately from Lemma~\ref{lem:basic-properties}(3) and Corollary~\ref{cor:extn-of-smc-funct-finite}, and Lemma~\ref{lem:exceptional}, respectively.
Only $w$-sincerity remains to be checked.

Suppose for a contradiction that $\bigcup_{i = 2}^{w+1} \Sigma^{w+1-i} \sS_i \subseteq \sX \cap \Sigma^w \sY$ is not $w$-sincere. 
We claim that if $E$ is an exceptional $\kk Q$-module, then $t_{[E]}$ lies in the parabolic subgroup generated by $\{t_{[S_a]} \in R \mid a \in \supp{E} \}$. Indeed, consider the wide subcategory $\cC_E$ generated by the set of simple $\kk Q$-modules $\{S_a \mid a \in \supp{E} \}$. The exceptional module $E$ lies in $\cC_E$ and it can be extended to a complete exceptional sequence $\sE$ in $\cC_E$ (see \cite[Lemma 1]{CB}). By the transitivity of the action of the braid group on the set of complete exceptional sequences (see \cite{CB} again), $\sE$ is obtained from the set $\{S_a \mid a \in \supp{\sE} \}$ ordered into a complete exceptional sequence in $\cC_E$, via a sequence of mutations corresponding to the braid group action. The claim then follows from \cite[\S 2.2]{IS}. 
As a consequence, we have that $u_2 \cdots u_{w+1}$ lies in the parabolic subgroup generated by $\{t_{[S_a]} \in R \mid a \in \supp{\bigcup_{i = 2}^{w+1} \sE_i} \}$. 
This subgroup is a proper parabolic subgroup since $\bigcup_{i = 2}^{w+1} \sE_i$ is not sincere, as  $\supp{\sS_i} = \supp{\cC_i} = \supp{\sE_i}$, for each $i$. This contradicts the fact that $\su$ is a positive $w$-noncrossing partition. Therefore, $\bigcup_{i = 2}^{w+1} \Sigma^{w+1-i} \sS_i \subseteq \sX \cap \Sigma^w \sY$ is indeed $w$-sincere. 

Since $\phi$ is the restriction of a bijection, it is clearly injective. It remains to check that $\phi$ is surjective. 
Let $\sS$ be a simple-minded collection contained in $\Fw$. By Lemma~\ref{lem:exceptional}, we can order the elements in $\sS$ into a (complete) exceptional sequence $\sE$ in which the cohomological degrees are weakly decreasing. For each $1\leq i \leq w+1$, let $\sE_i$  the subsequence consisting of the elements of cohomological degree $i-1$, and $\cC_i$ be the smallest wide subcategory of $\mod{\kk Q}$ containing $H(\sE_i)$, where $H \colon \Db(\kk Q) \to \mod{\kk Q}$ is the standard cohomology functor; see \eqref{cohomology}. Order the simple objects of $\cC_i$ into an exceptional sequence (cf.~{\cite[Theorem 3]{Ringel}}), and let $u_i$ be the product of the corresponding reflections respecting the order of the exceptional sequence, which gives a $T$-reduced expression for $u_i$. Then $\su = (u_{w+1}, \ldots, u_1)$ is a $w$-noncrossing partition and $\phi (\su) = \sS$. 

Finally, we must check that $\su$ is positive. Suppose $\su$ is not positive. Then $u_w \cdots u_1$ lies in a proper parabolic subgroup  $W_J$ generated by $J \subsetneq R$. For each $1 \leq i \leq w$, each reflection appearing in the $T$-reduced expression of $u_i$ above lies in $W_J$ by Theorem~\ref{thm:parabolic}. This means that for $1 \leq i \leq w$ the simple objects of $\cC_i$ do not have support at the vertices of $Q$ corresponding to the simple reflections at $R \setminus J$. Since the support of $H(\sE_i)$ coincides with that of the simple objects in $\cC_i$, it follows that $\sS \cap \sX \cap \Sigma^w \sY$ is not $w$-sincere, contradicting Theorem~\ref{thm:w-sincere}. Hence, $\su$ must be positive and $\phi$ is surjective.
\end{proof}

\appendix
\section{Reduction of simple-minded collections revisited} \label{sec:reduction}

\begin{center}
\small by Raquel Coelho Sim\~oes, David Pauksztello and Alexandra Zvonareva
\end{center}

In this appendix, we provide an alternative proof of \cite[Theorem 3.1]{Jin2}, which avoids using a Verdier localisation and is similar to the analogous result for simple-minded systems in \cite[Section 6]{CSP20}.

Throughout, $\sD$ will be a Hom-finite, Krull-Schmidt, $\kk$-linear triangulated category with shift functor $\Sigma \colon \sD \to \sD$. We will impose the following setup.

\begin{setup} \label{blanket-setup}
Let $\sS$ be an $\infty$-orthogonal collection of objects in $\sD$ and $\sZ$ a subcategory of $\sD$ satisfying the following conditions:
\begin{enumerate}
\item $\extn{\sS}$ is covariantly finite in ${}\orth (\Sigma^{< 0} \sS)$ and contravariantly finite in $(\Sigma^{> 0} \sS)\orth$.
\item for $d \in \sD$, we have $\Hom_\sD (d, \Sigma^{\ll 0} \sS) = 0$ and  $\Hom_\sD (\Sigma^{\gg 0} \sS, d) = 0$; and, 
\item $\sZ \coloneqq {}\orth (\Sigma^{\leq 0} \sS) \cap (\Sigma^{\geq 0} \sS)\orth$.
\end{enumerate}
\end{setup}

In fact, if $\sD = \Db(A)$ for a finite-dimensional algebra $A$, the stronger condition that $\extn{\sS}$ is functorially finite in $\Db(A)$ holds by Lemmas~\ref{lem:basic-properties}(3) and \ref{cor:extn-of-smc-funct-finite}.

We recall the following construction from \cite[Section 4]{CSP20} (see also \cite{Jin2, Nakaoka}):
\begin{itemize}
\item For an object $z \in \sZ$, a functor $\extn{1} \colon \sZ \to \sZ$ is defined on objects by taking the cone of a minimal right $\extn{\sS}$-approximation: $\tri{s_z}{\Sigma z}{z\extn{1}}$. The functor $\extn{1}$ is defined on morphisms in the obvious way, and its quasi-inverse $\extn{-1}$ by the dual construction; see \cite[Lemma 3.6]{CSP20}.
\item For a morphism $f \colon x \to y$ in $\sZ$, consider the triangle 
$\trilabels{x}{y}{c_f}{f}{g_1}{h_1}$ in $\sD$ together with the minimal right $\extn{\sS}$-approximation triangles of $c_f$ and $\Sigma x$ and complete to the commutative diagram below.
\[
\xymatrix@!R=8px{& & s_f \ar[d]^{\alpha_f} \ar[r]^{\sigma} & s_x \ar[d]^{\alpha_x} \\
x \ar[r]^f & y \ar[r]^{g_1} \ar[dr]_g & c_f \ar[r]^{h_1} \ar[d]^{\beta_f} & \Sigma x \ar[d]^{\beta_x} \\
& & z_f \ar[r]^h \ar[d]^{\gamma_f} & x \extn{1} \ar[d]^{\gamma_x} \\
& & \Sigma s_f \ar[r]_{\Sigma \sigma} & \Sigma s_x,}
\]
Note that applying $\Hom_\sD(\Sigma^{>0} \sS,-)$ to the triangle $\tri{x}{y}{c_f}$ shows that both $c_f, \Sigma x \in (\Sigma^{>0} \sS)\orth$ so that the required approximations exist. 
\end{itemize}

\begin{theorem}[{\cite[Proposition 3.6]{Jin2}}]
The category $(\sZ,\extn{1})$ admits a triangulated structure with standard triangles given by diagrams of the form $x \rightlabel{f} y \too z_f \too x\extn{1}$.
\end{theorem}

\begin{proof}
This can be proved in several ways. On one hand, observe that $\sS$ and $\sZ$ given in Setup~\ref{blanket-setup} satisfy \cite[Lemma 6.3]{CSP20} so that \cite[Theorem 4.1]{CSP20} can be applied. Otherwise, one can proceed via Jin's construction in \cite[Proposition 3.6]{Jin2}, or Nakaoka's construction in \cite[Theorem 4.15]{Nakaoka}.
\end{proof}

Because there are two triangulated structures under consideration, we differentiate them as in \cite[Section 6]{CSP20} with the following notation. 
Let $\sX$ and $\sY$ be subcategories of $\sZ$. We define 
\[
\sX \star \sY \coloneqq \{ z \in \sZ \mid 
\text{there exists a triangle } x \to z \to y \to x\extn{1} \text{ with } x \in \sX \text{ and } y \in \sY\}.
\]
We denote the extension closure of $\sX$ with respect to the triangulated structure in $\sZ$ by $\extnZ{\sX}$.
The notation $\sX * \sY$ and $\extn{\sX}$ keep their usual meanings in $\sD$.

Before giving our alternative proof of \cite[Theorem 3.1]{Jin2}, we collect some preliminary statements. 

\begin{lemma} \label{lem:prelim}
Let $\sS$ and $\sZ$ be as in Setup~\ref{blanket-setup}. Suppose one of the two conditions holds:
\begin{enumerate}[label=(\alph*)]
\item $\sS \subseteq \sT$ and $\sT$ is a simple-minded collection in $\sD$; in which case set $\sR = \sT \setminus \sS$; or
\item $\sT = \sS \cup \sR$, where $\sR$ is a simple-minded collection in $\sZ$.
\end{enumerate}
Then the following hold:
\begin{enumerate}
\item \label{thick-contain} For $i \geq j$, $\{\sR\extn{i}\} \star \{\sR\extn{i-1}\} \star \cdots \star \{\sR\extn{j}\} \subseteq \extn{\Sigma^i \sT} * \extn{\Sigma^{i-1} \sT} * \cdots * \extn{\Sigma^j \sT}$.
\item \label{susp-intersect} $\susp_\sD \sT \cap \sZ = \susp_\sZ \sR$.
\item \label{cosusp-intersect} $\cosusp_\sD \sT \cap \sZ = \cosusp_\sZ \sR$.
\item \label{susp-finite} $\susp_{\sD} \sS$ is contravariantly finite in $\sD$.
\item \label{cosusp-finite} $\cosusp_{\sD} \sS$ is covariantly finite in $\sD$.
\end{enumerate}
\end{lemma}

\begin{proof}
Statement \eqref{thick-contain} can be argued as in Claims A and B in the proof of \cite[Theorem 6.6]{CSP20}; see also \cite[Lemma 3.4]{Jin2}.
Statements \eqref{susp-intersect} and \eqref{cosusp-intersect} can be argued using \cite[Lemma 6.5]{CSP20} as the start of an induction.
Statements \eqref{susp-finite} and \eqref{cosusp-finite} follow immediately from conditions $(1)$ and $(2)$ in Setup ~\ref{blanket-setup} using a standard argument; cf.\ the proof of Proposition~\ref{prop:SMC-SMS-well-defined}.
\end{proof}

\begin{theorem}[{\cite[Theorem 3.1]{Jin2}}] \label{thm:SMCreduction}
Let $\sS$ and $\sZ$ be as in Setup~\ref{blanket-setup}. Then there is a bijection
\begin{align*}
\{ \text{simple-minded collections in $\sD$ containing $\sS$} \} & \bij  \{ \text{simple-minded collections in $\sZ$} \}. \\
\sT                                                                                               & \longmapsto  \sT \setminus \sS 
\end{align*}
\end{theorem}

\begin{proof}
Let $\sT$ be a simple-minded  collection in $\sD$ containing $\sS$. We want to show that $\sR \coloneqq \sT \setminus \sS$ is a simple-minded collection in $\sZ$. We will use the characterisation of simple-minded collections in  Proposition~\ref{prop:smc-char}. 

Firstly, observe that $\sR$ is an $\infty$-orthogonal collection in $\sZ$, for instance, by using the dimension shifting argument in the proof of \cite[Theorem 6.6]{CSP20}.

The next step is to show that ${}\orth (\sR \extn{< 0}) \cap (\sR \extn{\geq0})\orth = 0$. Suppose $z \in \sZ$ is such that $z \in {}\orth (\sR \extn{< 0}) \cap (\sR \extn{\geq0})\orth$. Since $\Hom_\sD (z \extn{>0}, \sR) \simeq \Hom_\sD (\Sigma^{>0} z, \sR)$ and $\Hom_\sD (\sR, z \extn{\leq 0}) \simeq \Hom_\sD (\sR, \Sigma^{\leq 0} z)$ by the dimension-shifting argument above, it follows that $z \in {}\orth (\Sigma^{<0} \sR) \cap (\Sigma^{\geq 0} \sR)\orth$. Since $z \in \sZ$, it follows that $z \in {}\orth (\Sigma^{<0} \sT) \cap (\Sigma^{\geq 0} \sT)\orth$, and so $z = 0$, as $\sT$ is $\infty$-Riedtmann in $\sD$ by Proposition~\ref{prop:smc-char}.

Finally, we show that $\cosusp_\sZ \sR$ is covariantly finite in $\sZ$.  
By Lemma~\ref{lem:prelim}\eqref{thick-contain}, we have $\cosusp_\sZ \sR \subseteq \cosusp_\sD \sT$. By Lemma~\ref{lem:prelim}\eqref{cosusp-intersect}, we have $(\cosusp_\sD \sT )\cap \sZ = \cosusp_\sZ \sR$. Since $\sT$ is a simple-minded collection in $\sD$, $\cosusp_\sD \sT$ is covariantly finite in $\sD$. Therefore, for $z \in \sZ \subseteq \sD$, we can take a decomposition triangle, $\trilabel{x}{z}{t_z}{f}$, in $\sD$ with $x \in {}\orth(\cosusp_\sD \sT) \subseteq {}\orth(\Sigma^{\leq 0} \sS)$. An inspection of the resulting long exact sequence shows that $x \in (\Sigma^{\geq 0} \sS)\orth$, and hence $x \in \sZ$. Now taking the cone of a minimal right $\extn{\sS}$-approximation $s_f \too t_z$ of $t_z$ gives a triangle in $\sZ$:
\[
x \rightlabel{f} z \too z_f \too x\extn{1}.
\]
We claim that the third term $z_f$ lies in  $(\cosusp_\sD \sT) \cap \sZ$ and hence in $\cosusp_\sZ \sR$, which makes the triangle above into a decomposition triangle for $z$ in $\sZ$ because $x \in {}\orth (\cosusp_\sD \sT) \subseteq {}\orth (\cosusp_\sZ \sR)$. In particular, showing that $\cosusp_\sZ \sR$ is covariantly finite in $\sZ$. The argument for the contravariant finiteness of $\susp_\sZ \sR$ in $\sZ$ is similar. Now applying Proposition~\ref{prop:smc-char}, we conclude that $\sR$ is a simple-minded collection in $\sZ$.

Therefore, we need to establish the claim above. Since $t_z \in \cosusp_\sD \sT$, there is a triangle $\tri{t'}{t_z}{t''}$ with $t' \in \extn{\sT}$ and $t'' \in \cosusp_\sD \Sigma^{-1} \sT$. Combining this triangle with the one coming from the minimal right $\extn{\sS}$-approximation gives the following commutative diagram by the octahedral axiom:
\[
\xymatrix@!R=8px{
                                                 & s_f \ar@{=}[r] \ar[d]_{g}  & s_f \ar[d]                      & \\
\Sigma^{-1} t'' \ar[r] \ar@{=}[d] & t'  \ar[r] \ar[d]            &  t_z \ar[r] \ar[d]  &  t'' \ar@{=}[d] \\
\Sigma^{-1} t'' \ar[r]                  & t_1 \ar[r] \ar[d]           & z_f \ar[r] \ar[d]              & t'' \\
                                                & \Sigma s_f \ar@{=}[r] & \Sigma s_f     &
}
\]
Examining the long exact sequence obtained by applying $\Hom_\sD(\sS,-)$ to the lower horizontal triangle shows that $g \colon s_f \to t'$ is a right $\extn{\sS}$-approximation. If it were not minimal, then $s_f \to t_z$ would also fail to be minimal. Therefore $g \colon s_f \to t'$ is a minimal right $\extn{\sS}$-approximation and $t_1 \in \extn{\sT}$ by \cite[Theorem 2.11]{CSP20}. Hence, $z_f \in \cosusp_\sD \sT$, as required.

Conversely, suppose $\sR$ is a simple-minded collection in $\sZ$. We will show that $\sT \coloneqq \sS \cup \sR$ is a simple-minded collection in $\sD$. The proof that $\sT$ is an $\infty$-orthogonal collection is similar to the corresponding argument above in the other direction. 

Let $d$ be an object in $\sD$. By Lemma~\ref{lem:prelim}\eqref{susp-finite} and \eqref{cosusp-finite}, we have that $\susp_\sD \sS$ is contravariantly finite in $\sD$ and $\cosusp_\sD \sS$ is covariantly finite in $\sD$. Decomposing first with respect to the t-structure $(\susp_\sD \sS, (\Sigma^{\geq 0} \sS)\orth)$ and then the t-structure $({}\orth(\Sigma^{\leq 0} \sS), \cosusp_\sD \sS)$ gives the following triangles,
\[
\tri{x}{d}{y} 
\quad \text{and} \quad
\tri{z}{y}{v},
\]
in which $z \in \sZ = {}\orth(\Sigma^{\leq 0} \sS) \cap (\Sigma^{\geq 0} \sS)\orth$, $v \in \cosusp_\sD \sS$ and $x \in \susp_\sD \sS$. This shows that $\sD = (\susp_\sD \sS) * \sZ * (\cosusp_\sD \sS)$, cf.\ \cite[Proposition 3.3]{Jin2}. Now, Lemma~\ref{lem:prelim}\eqref{thick-contain}, shows that $\sZ = \thick{\sZ}{\sR} \subseteq \thick{\sD}{\sT}$, from which it follows that $\sD = \thick{\sD}{\sT}$, and we conclude that $\sT$ is a simple-minded collection.
\end{proof}


\bigskip

{\small
\begin{flalign*}
\text{Email addresses: } & \text{\tt r.coelhosimoes@lancaster.ac.uk} & \\ 
                         & \text{\tt d.pauksztello@lancaster.ac.uk,} & \\ 
                         & \text{\tt david.ploog@uis.no}, & \\  
                         & \text{\tt alexandra.zvonareva@mathematik.uni-stuttgart.de} &
\end{flalign*}}

\end{document}